\documentclass[11pt]{article}
\usepackage{amsmath,mathdots}
\usepackage{amssymb}
\usepackage{latexsym}
\usepackage{amsthm}
\usepackage{amsmath,amssymb,amsfonts,amsthm}
\usepackage{mdframed}
\usepackage{graphicx}
\usepackage{stmaryrd}
\usepackage{ulem}
\usepackage{gensymb}
\usepackage{bm}

\usepackage{tikz}   
\usepackage{stmaryrd}
\usepackage{multirow}
\usepackage{amssymb}
\usepackage{mathtools}
\usetikzlibrary{decorations.pathmorphing}
\usepackage[margin=2.9cm]{geometry}
\usepackage{graphicx}
\usepackage{tikz}
\usetikzlibrary{matrix,arrows}
\usetikzlibrary{positioning}
\usetikzlibrary{fit}
\usetikzlibrary{patterns}
\usetikzlibrary{calc} 
\usetikzlibrary{positioning, arrows.meta}

\usepackage[colorlinks,
linkcolor=blue,
anchorcolor=blue,
citecolor=blue
]{hyperref}

\newtheorem{thm}{Theorem}[section]

\newtheorem{lem}[thm]{Lemma}

\newtheorem{exa}[thm]{Example}

\newcommand{\M}{\mathcal{M}}
\newcommand{\MM}{\widetilde{M}}

\newcommand{\red}{\mbox{\sf{redarc}}}

\newcounter{exacounter}

\usepackage{graphicx}
\usepackage{tikz}
\usetikzlibrary{matrix,arrows}
\usetikzlibrary{positioning}
\usetikzlibrary{fit}
\usetikzlibrary{patterns}

\begin{document}

\begin{center}
{\large \bf Stoimenow matchings avoiding multiple Catalan patterns simultaneously}
\end{center}

\begin{center}
Shuzhen Lv$^{a}$ and Sergey Kitaev$^{b}$
\\[6pt]

$^{a}$College of Mathematical Sciences \& Institute of Mathematics and Interdisciplinary Sciences, Tianjin Normal University, \\ Tianjin  300387, P. R. China\\[6pt]

$^{b}$Department of Mathematics and Statistics, University of Strathclyde, \\ 26 Richmond Street, Glasgow G1 1XH, UK\\[6pt]

Email:  $^{a}${\tt  lvshuzhenlsz@yeah.net},
           $^{b}${\tt sergey.kitaev@strath.ac.uk}
\end{center}

\noindent\textbf{Abstract.}
Motivated by Vassiliev’s knot invariants, Stoimenow introduced a special class of matchings, now known as Stoimenow matchings. These matchings have since been linked to various combinatorial structures enumerated by the Fishburn numbers. In a recent paper, a problem posed by Bevan et al.\ was addressed concerning the identification of subsets of Stoimenow matchings counted by the Catalan numbers. Five such subsets were presented, each defined by the avoidance of a single pattern, referred to as a Catalan pattern, within Stoimenow matchings.

In the present paper, we extend this line of research by enumerating all cases of simultaneous avoidance of sets of Catalan patterns in Stoimenow matchings. This comprehensive analysis reveals connections to nine integer sequences listed in the OEIS.\\[-3mm]



\section{Introduction}

A {\it matching} on the set of integers \(\{1, 2, \ldots, 2n\}\) is a partition of that set into blocks of size 2, called {\it arcs}. For an arc \([a, b]\), where \(a < b\), \(a\) (resp., \(b\)) is called the \textit{opener} (resp., \textit{closer}) of the arc. We represent matchings by drawing them with openers (resp., closers) arranged in increasing order from left to right. A matching is said to be {\it Stoimenow} if there
are no occurrences of Type 1 or Type 2 arcs:

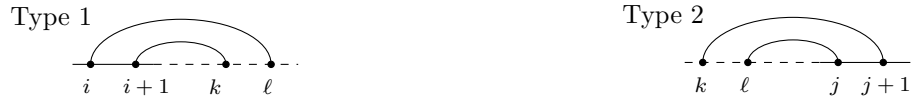
\begin{minipage}[b]{0.45\textwidth}
        \centering
        \begin{tikzpicture}[scale = 0.6]
    \draw (-0.4,0) -- (1.4,0);
    \draw [dashed](1.4,0) -- (4.6,0);
    \foreach \x in {0,1,3,4}
    {
        \filldraw (\x,0) circle (2pt);
    }
            \draw[black] (0,0) arc (180:0:2 and 1);
            \draw[black] (1,0) arc (180:0:1 and 0.5);
             \node[anchor=north west] at (-0.4, -0.1) {\scriptsize{$i$}};
             \node[anchor=north west] at (0.5, -0.1) {\scriptsize{$i+1$}};
             \node[anchor=north west] at (2.4, -0.1) {\scriptsize{$k$}};
             \node[anchor=north west] at (3.6, -0.1) {\scriptsize{$\ell$}};
             \node[anchor=north west] at (-2, 1.5) {{\small Type 1}};
        \end{tikzpicture}
        \vspace{0.2cm} 
        
    \end{minipage}
    \hfill 
    \begin{minipage}[b]{0.45\textwidth}
        \centering
        \begin{tikzpicture}[scale = 0.6]
           \draw[dashed] (-0.4,0) -- (2.6,0);
    \draw (2.6,0) -- (4.6,0);
    \foreach \x in {0,1,3,4}
    {
        \filldraw (\x,0) circle (2pt);
    }
            \draw[black] (0,0) arc (180:0:2 and 1);
            \draw[black] (1,0) arc (180:0:1 and 0.5);
             \node[anchor=north west] at (-0.4, -0.1) {\scriptsize{$k$}};
             \node[anchor=north west] at (0.6, -0.1) {\scriptsize{$\ell$}};
             \node[anchor=north west] at (2.6, -0.1) {\scriptsize{$j$}};
             \node[anchor=north west] at (3.3, -0.1) {\scriptsize{$j+1$}};
             \node[anchor=north west] at (-2, 1.5) {{\small Type 2}};
        \end{tikzpicture}
        \vspace{0.2cm}
    \end{minipage}

\vspace{-2mm}
\noindent
Stoimenow matchings are precisely the {\it regular linearized chord diagrams} introduced in~\cite{stoim}.
They are enumerated by the {\it Fishburn numbers} (OEIS~A022493~{\rm \cite{OEIS}}) and correspond bijectively to {\it {\bf (2+2)}-free posets}, {\it ascent sequences}, {\it Fishburn permutations} and {\it Fishburn matrices}~\cite{Bousquet-Claesson-Dukes-Kitaev-DukesParvi,DukPar}.

Let $\M_n$ be the set of all Stoimenow matchings with $n$ arcs, and denote its cardinality by  $|\M_n|$.
 We assume that a matching $M\in\M_n$ is formed by arcs $\{[a_i,b_i]\}_{i=1}^{n}$ and can be written as $M=\{[a_1,b_1],[a_2,b_2],\dots,[a_n,b_n]\}$.
A subset of $M$ of the form $\{[a_1,b_1],[a_2,b_2],\ldots,[a_k,b_k]\}$ with $k\geq 0$ is called (i) a $k$-{\it crossing} if $a_1<a_2<\cdots<a_k<b_1<b_2<\cdots <b_k$; (ii) a $k$-{\it noncrossing} if $a_1<b_1<a_2<b_2<\cdots<a_k<b_k$.

Bevan et al.~\cite{Bevan-Cheon-Kitaev} presented a hierarchy related to the Fishburn numbers (see \cite[Fig.~12]{Bevan-Cheon-Kitaev}). In connection with this hierarchy, they posed the problem of describing a subset of Stoimenow matchings counted by the {\it Catalan numbers} \(C_n = \frac{1}{n+1}\binom{2n}{n}\). This problem was solved in several ways in \cite{LvKitZhang}. Specifically, it was shown that avoidance of any of the patterns \(P_1\)--\(P_5\) introduced below in Stoimenow matchings yields objects counted by the Catalan numbers; these patterns are therefore called {\it Catalan patterns}. By avoidance of a pattern in a matching, we mean that the matching does not contain any submatching isomorphic to the pattern. \\[-1mm]
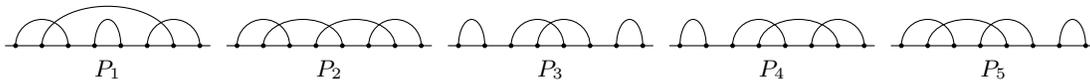
 \[
\begin{tikzpicture}[scale = 0.35] 
        \node at (3.5,-0.9) {\upshape {\footnotesize{$P_1$}}};
        \draw (-0.4,0) -- (7.4,0);
        \foreach \x in {0,1,2,3,4,5,6,7}
        {
            \filldraw (\x,0) circle (2pt);
        }
        \draw[black] (0,0) arc (180:0:1 and 1);
        \draw[black] (1,0) arc (180:0:2.5 and 1.5);
        \draw[black] (3,0) arc (180:0:0.5 and 1);
        \draw[black] (5,0) arc (180:0:1 and 1);
    \end{tikzpicture}
    \hspace{0.2cm}
    \begin{tikzpicture}[scale = 0.35] 
        \node at (3.5,-0.9) {\upshape {\footnotesize{$P_2$}}};
        \draw (-0.4,0) -- (7.4,0);
        \foreach \x in {0,1,2,3,4,5,6,7}
        {
            \filldraw (\x,0) circle (2pt);
        }
        \draw[black] (0,0) arc (180:0:1 and 1);
        \draw[black] (1,0) arc (180:0:1.5 and 1);
        \draw[black] (3,0) arc (180:0:1.5 and 1);
        \draw[black] (5,0) arc (180:0:1 and 1);
    \end{tikzpicture}
    \hspace{0.2cm}
    \begin{tikzpicture}[scale = 0.35] 
        \node at (3.5,-0.9) {\upshape {\footnotesize{$P_3$}}};
        \draw (-0.4,0) -- (7.4,0);
        \foreach \x in {0,1,2,3,4,5,6,7}
        {
            \filldraw (\x,0) circle (2pt);
        }
        \draw[black] (0,0) arc (180:0:0.5 and 1);
        \draw[black] (2,0) arc (180:0:1 and 1);
        \draw[black] (3,0) arc (180:0:1 and 1);
        \draw[black] (6,0) arc (180:0:0.5 and 1);
    \end{tikzpicture}
   \hspace{0.2cm}
    \begin{tikzpicture}[scale = 0.35] 
        \node at (3.5,-0.9) {\upshape {\footnotesize{$P_4$}}};
        \draw (-0.4,0) -- (7.4,0);
        \foreach \x in {0,1,2,3,4,5,6,7}
        {
            \filldraw (\x,0) circle (2pt);
        }
        \draw[black] (0,0) arc (180:0:0.5 and 1);
        \draw[black] (2,0) arc (180:0:1 and 1);
        \draw[black] (3,0) arc (180:0:1.5 and 1);
        \draw[black] (5,0) arc (180:0:1 and 1);
    \end{tikzpicture}
    \hspace{0.2cm}
    \begin{tikzpicture}[scale = 0.35] 
        \node at (3.5,-0.9) {\upshape {\footnotesize{$P_5$}}};
        \draw (-0.4,0) -- (7.4,0);
        \foreach \x in {0,1,2,3,4,5,6,7}
        {
            \filldraw (\x,0) circle (2pt);
        }
        \draw[black] (0,0) arc (180:0:1 and 1);
        \draw[black] (1,0) arc (180:0:1.5 and 1);
        \draw[black] (3,0) arc (180:0:1 and 1);
        \draw[black] (6,0) arc (180:0:0.5 and 1);
    \end{tikzpicture}
\]\\[-1.2cm]

In Sections~\ref{pairs-sec}--\ref{5-sec}, we enumerate all cases of multi-avoidance involving the patterns $P_1$--$P_5$, revealing numerous connections to known sequences in the OEIS~{\rm \cite{OEIS}}; see Table~\ref{tab-results} for a summary. In these sections, we denote
\[
A(x) = \sum_{n \geq 0} |\mathcal{M}_n(S)|\, x^n,
\]
where $S$ is the set of patterns under consideration, and $|\mathcal{M}_n(S)|$ is the cardinality of $\mathcal{M}_n(S)$, the set of matchings in $\mathcal{M}_n$ avoiding each pattern in $S$ simultaneously. We let $\M(S)=\cup_{n\geq0}\M_n(S)$.

\begin{table}
 	{
 		\renewcommand{\arraystretch}{1}
 		\setlength{\tabcolsep}{5pt}
 \begin{center} 
 		\begin{tabular}{|c| c| c| c| c| }
 			\hline
 		{\scriptsize Patterns}  & \scriptsize{Sequence ($n\geq 1$)}	& {\scriptsize OEIS}  & \scriptsize{G.f.}	& {\scriptsize Reference} \\
 			\hline
 			\hline
& &   && \\[-4mm]
\scriptsize{$(P_1, P_2)$} & 
\scriptsize{1,2,5,13,33,82,202,497,\ldots} & \scriptsize{A116703}   &
\scriptsize{$
    \frac{(1-x)^3}{1-4x+5x^2-3x^3}
$} & \scriptsize{\multirow{1.2}{*}{Theorem~\ref{thm-P1-P2}}}
\\[1mm]
\hline

& &   && \\[-4mm]
 \scriptsize{$(P_1, P_3)$} & \scriptsize{1,2,5,13,32,73,156,318,\ldots} & \scriptsize{New}  &
\scriptsize{$\frac{1-6x+15x^2-19x^3+13x^4-5x^5}{(1-x)^5(1-2x)}$} & \scriptsize{Theorem~\ref{thm-P1-P3}}
\\[1mm]
\hline

 \scriptsize{$(P_1, P_4)$} & \scriptsize{\multirow{3.6}{*}{1,2,5,13,33,81,193,449,\ldots}} & \scriptsize{\multirow{3.6}{*}{A005183}}   &
\scriptsize{\multirow{3.4}{2cm}{$\frac{1-4x+5x^2-x^3}{(1-x)(1-2x)^2}$}} & \scriptsize{\multirow{3.6}{*}{Theorem~\ref{thm-A005183}}}
\\
\cline{1-1}

 \scriptsize{$(P_1, P_5)$} & &   && 
\\
\cline{1-1}

 \scriptsize{$(P_2, P_3)$} & &   && 
\\
\hline

& &   && \\[-4mm]
  \scriptsize{$(P_2, P_4)$} & \scriptsize{\multirow{6}{*}{1,2,5,13,34,89,233,610,\ldots}} & \scriptsize{\multirow{6}{*}{A001519}}   &
\scriptsize{\multirow{5}{1.2cm}{$\frac{1-2x}{1-3x+x^2}$}} & \scriptsize{\multirow{6}{*}{Theorem~\ref{thm-A001519}}}
\\
\cline{1-1}

 \scriptsize{$(P_2, P_5)$} & &   && 
\\
\cline{1-1}

 \scriptsize{$(P_3, P_4)$} & &   && 
\\
\cline{1-1}

 \scriptsize{$(P_3, P_5)$} & &   && 
\\
\cline{1-1}

 \scriptsize{$(P_4, P_5)$} & &   && 
\\
\hline
& &   && \\[-4mm]
 \scriptsize{$(P_1, P_2, P_3)$} & \scriptsize{1,2,5,12,25,47,82,135,\ldots} & \scriptsize{A116722}   &
 \scriptsize{$\frac{1-4x+7x^2-5x^3+2x^4-x^5+x^6}{(1-x)^5}$} 
& \scriptsize{Theorem~\ref{thm-P1-P2-P3}}
\\[1mm]
\hline

 \scriptsize{$(P_1, P_2, P_4)$} & \scriptsize{\multirow{6.5}{*}{1,2,5,12,27,58,121,248,\ldots}} & \scriptsize{\multirow{6.5}{*}{A000325}}   &
\scriptsize{\multirow{6}{1.6cm}{$\frac{1-3x+3x^2}{(1-x)^2(1-2x)}$}} & \scriptsize{\multirow{6.5}{*}{Theorem~\ref{thm-A000325}}}
\\
\cline{1-1}

 \scriptsize{$(P_1, P_2, P_5)$} & &   && 
\\
\cline{1-1}

 \scriptsize{$(P_2, P_3, P_4)$} & &   && 
\\
\cline{1-1}

 \scriptsize{$(P_2, P_3, P_5)$} & &   && 
\\
\cline{1-1}

 \scriptsize{$(P_1, P_4, P_5)$} & &   && 
\\
\cline{1-1}

 \scriptsize{$(P_2, P_4, P_5)$} & &   && 
\\
\hline

 \scriptsize{$(P_1, P_3, P_4)$} & \scriptsize{\multirow{2.5}{*}{1,2,5,12,26,99,184,340,\ldots}} & \scriptsize{\multirow{2.5}{*}{A116725}}   &
\scriptsize{\multirow{2.5}{3.3cm}{$\frac{1-5x+10x^2-9x^3+3x^4-x^5}{(1-x)^4(1-2x)}$}} & \scriptsize{\multirow{2.5}{*}{Theorem~\ref{thm-A116725}}}
\\
\cline{1-1}

 \scriptsize{$(P_1, P_3, P_5)$} & &   && 
\\
\hline

 & &   && \\[-4mm]
\scriptsize{$(P_3, P_4, P_5)$} & \scriptsize{1,2,5,12,28,65,151,351,\ldots} & \scriptsize{A034943}   &
\scriptsize{$\frac{(1-x)^2}{1-3x+2x^2-x^3}$} 
& \scriptsize{Theorem~\ref{thm-P3-P4-P5}}
\\[1mm]
\hline

 \scriptsize{$(P_1, P_2, P_3, P_4)$} & \scriptsize{\multirow{3.6}{*}{1,2,5,11,21,36,57,85,\ldots}} & \scriptsize{\multirow{3.6}{*}{A050407}}   &
 \scriptsize{\multirow{3.6}{2cm}{$\frac{1-3x+4x^2-x^3}{(1-x)^4}$}} 
& \scriptsize{\multirow{3.6}{*}{Theorem~\ref{thm-A050407}}}
\\
\cline{1-1}

 \scriptsize{$(P_1, P_2, P_3, P_5)$} & &   && 
\\
\cline{1-1}

 \scriptsize{$(P_2, P_3, P_4, P_5)$} & &   && 
\\
\hline

 \scriptsize{$(P_1, P_2, P_4, P_5)$} & \scriptsize{\multirow{2.5}{*}{1,2,5,11,22,42,79,149,\ldots}} & \scriptsize{\multirow{2.5}{*}{New}}   &
 \scriptsize{\multirow{2.5}{2.5cm}{$\frac{1-4x+6x^2-3x^3-x^4}{(1-x)^3(1-2x)}$}}
& \scriptsize{\multirow{2.5}{*}{Theorem~\ref{thm-P1-P2-P4-P5, P1-P3-P4-P5}}}
\\
\cline{1-1}

 \scriptsize{$(P_1, P_3, P_4, P_5)$} & &   && 
\\
\hline
& &   && \\[-4mm]
 \scriptsize{$(P_1, P_2, P_3, P_4, P_5)$} & \scriptsize{1,2,5,10,17,26,37,50,\ldots} & \scriptsize{A002522}   &
\scriptsize{$\frac{1-2x+2x^2+x^3}{(1-x)^3}$} 
& \scriptsize{Theorem~\ref{thm-P1-P2-P3-P4-P5}}
\\[1mm]
\hline
	\end{tabular}
\end{center} 
}
\vspace{-0.5cm}
 	\caption{Stoimenow matchings avoiding multiple patterns.}\label{tab-results}
\end{table}

\section{Preliminaries}

In this section, we provide the essential technical details from~\cite{LvKitZhang} needed to derive the results of this paper; for further details, we refer the reader to~\cite{LvKitZhang}. We also note that in~\cite{LvKitZhang}, the results for the pattern $P_3$ are obtained as a special case of the more general results for the pattern $P^k_3$ (not defined here). Therefore, to establish our results on multi-avoidance, we need to present in Section~\ref{subsec-P3-avoid} an alternative derivation of the fact that $|\M_n(P_3)|=C_n$.

Given a matching $M \in \M_n$ with an arc $[a_n,2n]$ (where $a_n$ is the last opener), the {\it reduction arc} of $M$ is $\red(M)=[a_j,a_n+1]$, introduced in~\cite{Claesson-Dukes-Kitaev}. As shown in Figure~\ref{fig-exa-glue}, for the matching $M_1=\{[1,4],[2,5],[3,8],[6,9],[7,10]\}$, we have
$\red(M_1)=[3,8]$, while for $M_2=\{[1,2],[3,5], [4,6],[7,8]\}$, we have $\red(M_2)=[7,8]$.

A matching \( M \in \mathcal{M}_n \) is \textit{reducible} if there exists an integer \( i \), with \( 1 < i < 2n \), such that the number of openers equals the number of closers in the set \( \{1, 2, \ldots, i\} \). Otherwise, \( M \) is \textit{irreducible}. 
Referring to Figure~\ref{fig-exa-glue}, the matching $M_1$ is irreducibel, while the matching $M_2$ is reducible.
Each matching can be decomposed into {\it irreducible blocks}. 
For example, the matching $M_2$ as shown in Figure~\ref{fig-exa-glue} can be decomposed into the blocks $B_1=\{[1,2]\}$, $B_2=\{[3,5], [4,6]\}$, and $B_3=\{[7,8]\}$.

\subsection{\( P_2 \)-avoiding Stoimenow matchings}\label{subsec-P2-avoiding}

Given a pair $M_1\in\mathcal{M}_{k-1}(P_2)$ and $M_2\in\mathcal{M}_{n-k}(P_2)$ with $1\le k\le n$, the {\it gluing procedure} that combines $M_1$ and $M_2$ into a matching $M\in\mathcal{M}_n(P_2)$ is defined as follows, following~\cite{LvKitZhang}. We first obtain $\widetilde{M}$ from~$M_1$. 
If $M_1=\emptyset$, set $\widetilde{M}=\{[1,2]\}$; otherwise, let $M_1=\{[a_1,b_1],\dots,[a_{k-1},b_{k-1}]\}$ and form $\widetilde{M}$ by adding a reduction arc $[a,b]$, where the opener $a$ can be placed in one of two possible positions:
\begin{itemize}
\item  If \( M_1 \) is a \( (k-1) \)-crossing, we place \( a \) immediately to the left of the first opener in \( M_1 \).
\item If $M_1$ is not a $(k-1)$-crossing, let $a_i,a_{i+1},\dots,a_j$ be all openers less than $b_1$ whose corresponding closers are greater than $a_{k-1}$.
Then place $a$ immediately before $X$ if $X$ is nonempty, or before $b_1$ otherwise.
\end{itemize}
Finally, \( M \) is obtained by merging \(\MM\) and  \(M_2\) after increasing each element in \(M_2\) by $2k$. 

For $M\in\mathcal{M}_n(P_2)$, following~\cite{LvKitZhang}, the {\it splitting procedure} for splitting~$M$ into $M_1$ and $M_2$ is as follows.
Let $\MM$ be the first irreducible block of~$M$, and let $M_2$ be the matching formed by the remaining arcs after decreasing every element by~$|\MM|$.
The matching $M_1$ is obtained from $\MM$ by deleting its reduction arc.

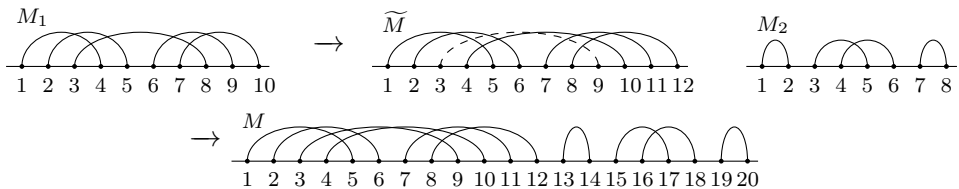
\begin{figure}[h]
    \centering
     \begin{tikzpicture}[scale = 0.35] 
        \draw (-0.6,0) -- (9.4,0);
        \foreach \x in {0,1,2,3,4,5,6,7,8,9}
        {
            \filldraw (\x,0) circle (2pt);
        }
        \draw[black] (0,0) arc (180:0:1.5 and 1.3);
        \draw[black] (1,0) arc (180:0:1.5 and 1.3);
        \draw[black] (2,0) arc (180:0:2.5 and 1.3);
        \draw[black] (5,0) arc (180:0:1.5 and 1.3);
        \draw[black] (6,0) arc (180:0:1.5 and 1.3);
        
        \node[anchor=north west] at (-0.6, 2.6) {\scriptsize $M_1$};
     \node[anchor=north west] at (-0.6, -0.05) {\scriptsize{1}};
         \node[anchor=north west] at (0.4, -0.05) {\scriptsize{2}};
         \node[anchor=north west] at (1.4, -0.05) {\scriptsize{3}};
          \node[anchor=north west] at (2.4, -0.05) {\scriptsize{4}};
         \node[anchor=north west] at (3.4, -0.05) {\scriptsize{5}};
         \node[anchor=north west] at (4.4, -0.05) {\scriptsize{6}};
         \node[anchor=north west] at (5.4, -0.05) {\scriptsize{7}};
         \node[anchor=north west] at (6.4, -0.05) {\scriptsize{8}};
         \node[anchor=north west] at (7.4, -0.05) {\scriptsize{9}};
          \node[anchor=north west] at (8.4, -0.05) {\scriptsize{10}};
    \end{tikzpicture}
\hspace{0.1cm}
\raisebox{4ex}{$\xrightarrow{\hspace*{0.2cm}}$}  
\hspace{0.1cm}
    \begin{tikzpicture}[scale = 0.35] 
        \draw (-0.6,0) -- (11.4,0);
        \foreach \x in {0,1,2,3,4,5,6,7,8,9,10,11}
        {
            \filldraw (\x,0) circle (2pt);
        }
        \draw[black] (0,0) arc (180:0:2 and 1.3);
        \draw[black] (1,0) arc (180:0:2 and 1.3);
        \draw[black,dashed] (2,0) arc (180:0:3 and 1.3);
        \draw[black] (3,0) arc (180:0:3 and 1.3);
        \draw[black] (6,0) arc (180:0:2 and 1.3);
        \draw[black] (7,0) arc (180:0:2 and 1.3);  
        
        \node[anchor=north west] at (-0.6,2.5) {\scriptsize $\widetilde{M}$};        
     \node[anchor=north west] at (-0.6, -0.05) {\scriptsize{1}};
         \node[anchor=north west] at (0.4, -0.05) {\scriptsize{2}};
         \node[anchor=north west] at (1.4, -0.05) {\scriptsize{3}};
          \node[anchor=north west] at (2.4, -0.05) {\scriptsize{4}};
         \node[anchor=north west] at (3.4, -0.05) {\scriptsize{5}};
         \node[anchor=north west] at (4.4, -0.05) {\scriptsize{6}};
         \node[anchor=north west] at (5.4, -0.05) {\scriptsize{7}};
         \node[anchor=north west] at (6.4, -0.05) {\scriptsize{8}};
         \node[anchor=north west] at (7.4, -0.05) {\scriptsize{9}};
          \node[anchor=north west] at (8.4, -0.05) {\scriptsize{10}};
         \node[anchor=north west] at (9.4, -0.05) {\scriptsize{11}};
         \node[anchor=north west] at (10.4, -0.05) {\scriptsize{12}};
    \end{tikzpicture}
    \hspace{0.3cm}
    \begin{tikzpicture}[scale = 0.35] 
        \draw (-0.6,0) -- (7.4,0);
        \foreach \x in {0,1,2,3,4,5,6,7}
        {
            \filldraw (\x,0) circle (2pt);
        }
        \draw[black] (0,0) arc (180:0:0.5 and 1);
        \draw[black] (2,0) arc (180:0:1 and 1);
        \draw[black] (3,0) arc (180:0:1 and 1);
        \draw[black] (6,0) arc (180:0:0.5 and 1);
        \node[anchor=north west] at (-0.6, 2.3) {\scriptsize $M_2$};
        \node[anchor=north west] at (-0.6, -0.05) {\scriptsize{1}};
         \node[anchor=north west] at (0.4, -0.05) {\scriptsize{2}};
         \node[anchor=north west] at (1.4, -0.05) {\scriptsize{3}};
          \node[anchor=north west] at (2.4, -0.05) {\scriptsize{4}};
         \node[anchor=north west] at (3.4, -0.05) {\scriptsize{5}};
         \node[anchor=north west] at (4.4, -0.05) {\scriptsize{6}}; 
          \node[anchor=north west] at (5.4, -0.05) {\scriptsize{7}};
         \node[anchor=north west] at (6.4, -0.05) {\scriptsize{8}};
    \end{tikzpicture}
    \\
    \raisebox{4ex}{$\xrightarrow{\hspace*{0.2cm}}$}  
    \vspace{0.4cm}
    \begin{tikzpicture}[scale = 0.35, shift={(0,1)}] 
        \draw (-0.6,0) -- (19.4,0);
        \foreach \x in {0,1,2,3,4,5,6,7,8,9,10,11,12,13,14,15,16,17,18,19}
        {
            \filldraw (\x,0) circle (2pt);
        }
        \draw[black] (0,0) arc (180:0:2 and 1.3);
        \draw[black] (1,0) arc (180:0:2 and 1.3);
        \draw[black] (2,0) arc (180:0:3 and 1.3);
        \draw[black] (3,0) arc (180:0:3 and 1.3);
        \draw[black] (6,0) arc (180:0:2 and 1.3);
        \draw[black] (7,0) arc (180:0:2 and 1.3);
        \draw[black] (12,0) arc (180:0:0.5 and 1.3);
        \draw[black] (14,0) arc (180:0:1 and 1.3);
        \draw[black] (15,0) arc (180:0:1 and 1.3);
        \draw[black] (18,0) arc (180:0:0.5 and 1.3);
        
         \node[anchor=north west] at (-0.6,2.2) {\scriptsize $M$};
         \node[anchor=north west] at (-0.6, -0.05) {\scriptsize{1}};
         \node[anchor=north west] at (0.4, -0.05) {\scriptsize{2}};
         \node[anchor=north west] at (1.4, -0.05) {\scriptsize{3}};
          \node[anchor=north west] at (2.4, -0.05) {\scriptsize{4}};
         \node[anchor=north west] at (3.4, -0.05) {\scriptsize{5}};
         \node[anchor=north west] at (4.4, -0.05) {\scriptsize{6}};
         \node[anchor=north west] at (5.4, -0.05) {\scriptsize{7}};
         \node[anchor=north west] at (6.4, -0.05) {\scriptsize{8}};
         \node[anchor=north west] at (7.4, -0.05) {\scriptsize{9}};
          \node[anchor=north west] at (8.2, -0.05) {\scriptsize{10}};
         \node[anchor=north west] at (9.2, -0.05) {\scriptsize{11}};
         \node[anchor=north west] at (10.2, -0.05) {\scriptsize{12}};
         \node[anchor=north west] at (11.2, -0.05) {\scriptsize{13}};
          \node[anchor=north west] at (12.2, -0.05) {\scriptsize{14}};
         \node[anchor=north west] at (13.2, -0.05) {\scriptsize{15}};
         \node[anchor=north west] at (14.2, -0.05) {\scriptsize{16}};
         \node[anchor=north west] at (15.2, -0.05) {\scriptsize{17}};
         \node[anchor=north west] at (16.2, -0.05) {\scriptsize{18}};
         \node[anchor=north west] at (17.2, -0.05) {\scriptsize{19}};
         \node[anchor=north west] at (18.2, -0.05) {\scriptsize{20}};
    \end{tikzpicture}
    \vspace{-0.5cm}
    \caption{An example of gluing $M_1\in\mathcal{M}_5(P_2)$ and $M_2\in \mathcal{M}_4(P_2)$.}
    \label{fig-exa-glue}
\end{figure}

Figure~\ref{fig-exa-glue} illustrates an example of the gluing procedure, and the reverse steps yield the splitting procedure for \( M \).

\subsection{\( P_3 \)-avoiding Stoimenow matchings}\label{subsec-P3-avoid}

In this subsection, we give an alternative proof of the following result established in \cite{LvKitZhang}.

\begin{thm}\label{MnP3Cn}
For $n\geq 0$, $|\M_n(P_3)|=C_n$.
\end{thm}

\begin{proof}
We shall show that the g.f.\ $F(x)=\sum_{n\geq 0}|\M_n(P_3)|x^n$ coincides with 
$$C(x)=\sum_{n\geq 0}C_nx^n=\frac{1-\sqrt{1-4x}}{2x},$$
 to prove the theorem. Let $M=\{[1,a_1],\ldots,[a_n,2n]\}\in \M_n(P_3)$, we consider three cases that result in formula (\ref{F(x)-P3-formula}) below for $F(x)$. \\[-3mm]

\noindent
{\bf Case 1.}  The arcs $[1,b_1]$ and $[a_n,2n]$ overlap, so that $M$ is an $n$-crossing.
The contribution to $F(x)$ in this case is $\frac{1}{1-x}$, which also counts the empty matching.  \\[-3mm]

\noindent
{\bf Case 2.} The arcs $[1,b_1]$ and $[a_n,2n]$ do not overlap, and all openers in $M$ are between $1$ and $b_1$, and all closers in $M$ are between  $a_n$ and $2n$. Pattern $P_3$ cannot appear in this case, but to avoid Type 1 and Type 2 arcs, $M$ must be of the form presented schematically in Figure~\ref{case2-P3}. In that figure, the possibly empty run of openers $Y$ (resp., closers $Z$) has closers (resp., openers) between $b_1$ and $a_n$, while the possibly empty run of openers $(b_1-i)(b_1-i+1)\cdots (b_1-1)$ has closers $(a_n+1)(a_n+2)\cdots (a_n+i)$. Moreover, closers of arcs with openers in $Y$ and openers of arcs with closers in $Z$ can be shuffled arbitrarily without creating a forbidden configuration. The contribution to $F(x)$ in this case is $$\frac{x^2}{(1-x)(1-2x)},$$
where $\frac{1}{1-x}$ represents possibilities for openers not in $Y$, $x^2$ represents the arcs $[1,b_1]$ and $[a_n,2n]$, and $\frac{1}{1-2x}$ represents the choices between $b_1$ and $a_n$ (which are are given by binary strings of any length). \\[-3mm]

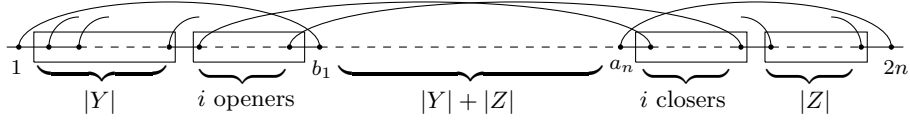
\begin{figure}
 \centering
    \begin{tikzpicture}[scale = 0.4] 
        \draw (-0.4,0) -- (2,0);
        \draw [dashed](2,0) -- (5,0);
        \draw (5,0) -- (6,0);
        \draw [dashed](6,0) -- (9,0);
        \draw (9,0) -- (10,0);
        \draw [dashed](10,0) -- (20,0);
        \draw (20,0) -- (21,0);
        \draw[dashed] (21,0) -- (24,0);
        \draw (24,0) -- (25,0);
        \draw[dashed] (25,0) -- (28,0);
        \draw (28,0) -- (29.4,0);
        \foreach \x in {0,1,2,5,6,9,10,20,21,24,25,28,29}
        {
            \filldraw (\x,0) circle (2pt);
        }
        \draw(0.5,-0.5)--(0.5,0.5)--(5.2,0.5)--(5.2,-0.5)--(0.5,-0.5);
        \draw(5.8,-0.5)--(5.8,0.5)--(9.5,0.5)--(9.5,-0.5)--(5.8,-0.5);
        \draw(20.5,-0.5)--(20.5,0.5)--(24.2,0.5)--(24.2,-0.5)--(20.5,-0.5);
        \draw(24.8,-0.5)--(24.8,0.5)--(28.2,0.5)--(28.2,-0.5)--(24.8,-0.5);

    \draw[domain=0:90, smooth, variable=\t, shift={(2,0)}] 
    plot ({-cos(\t)}, {sin(\t)});
    \draw[domain=0:90, smooth, variable=\t, shift={(3,0)}] 
    plot ({-cos(\t)}, {sin(\t)});
    \draw[domain=0:90, smooth, variable=\t, shift={(6,0)}] 
    plot ({-cos(\t)}, {sin(\t)});

\draw[domain=0:90, smooth, variable=\t, shift={(24,0)}] 
    plot ({cos(\t)}, {sin(\t)});
\draw[domain=0:90, smooth, variable=\t, shift={(27,0)}] 
    plot ({cos(\t)}, {sin(\t)});
        
        \draw[black] (0,0) arc (180:0:5 and 1.5);
        \draw[black] (20,0) arc (180:0:4.5 and 1.5);
        \draw[black] (6,0) arc (180:0:7.5 and 1.5);
        \draw[black] (9,0) arc (180:0:7.5 and 1.5);

        \node[anchor=north west] at (-0.6, -0.1) {\scriptsize{1}};
        \node[anchor=north west] at (9.4, -0.1) {\scriptsize{$b_1$}};
        \node[anchor=north west] at (19.3, -0.1) {\scriptsize{$a_n$}};
        \node[anchor=north west] at (28.4, -0.1) {\scriptsize{$2n$}};

     \node[anchor=north west] at (0.3, -0.2) {$\underbrace{ \hspace{1.7cm}}$};
     \node[anchor=north west] at (1.8, -1.1) {\footnotesize{$|Y|$}};
     \node[anchor=north west, rotate=90] at (5.65, -1.7) {$\Bigg \{$};
     \node[anchor=north west] at (5.6, -1.1) {\footnotesize{$i$ openers}};
     \node[anchor=north west] at (10.3, -0.2) {$\underbrace{ \hspace{3.5cm}}$};
     \node[anchor=north west] at (13, -1.1) {\footnotesize{$|Y|+|Z|$}};
      \node[anchor=north west, rotate=90] at (20.4, -1.7) {$\Bigg \{$};
     \node[anchor=north west] at (20.3, -1.1) {\footnotesize{$i$ closers}};
     \node[anchor=north west, rotate=90] at (24.5, -1.7) {$\Bigg \{$};
     \node[anchor=north west] at (25.6, -1.1) {\footnotesize{$|Z|$}};
    \end{tikzpicture}
\caption{The structure of $P_3$-avoiding Stoimenow matchings in Case 2.}\label{case2-P3}
\end{figure} 

\noindent
{\bf Case 3.} The arcs \([1, b_1]\) and \([a_n, 2n]\) do not overlap, and there exists an arc \([a_i, b_i]\) such that \(b_1 < a_i < b_i < a_n\). To avoid \(P_3\) as well as Type 1 and Type 2 arcs, \(M\) must have the structure shown schematically in Figure~\ref{case3-P3}, where \(A\), \(B\), and \(E\) must be present to distinguish this case from the previous ones. Details of this structure are as follows, where we let $\alpha$ (resp., $\beta$) be the label of a closer (resp., opener) with the opener (resp., closer) in $[1,b_1]$ (resp., $[a_n,2n]$):\\[-3mm]

\begin{figure}[h]
 \centering
    \begin{tikzpicture}[scale = 0.4] 
        \draw [dashed](-0.4,0) -- (4,0);
        \draw (4,0) -- (5,0);
        \draw [dashed](5,0) -- (22,0);
        \draw (22,0) -- (23,0);
        \draw [dashed](23,0) -- (27.4,0);
        \foreach \x in {0,2,4,5,7,9,11,13,18,20,22,23,25,27}
        {
            \filldraw (\x,0) circle (2pt);
        }
        \draw(1.8,-0.5)--(1.8,0.5)--(4.8,0.5)--(4.8,-0.5)--(1.8,-0.5);
       \draw(22.2,-0.5)--(22.2,0.5)--(25.2,0.5)--(25.2,-0.5)--(22.2,-0.5);
        
        \draw[black] (0,0) arc (180:0:2.5 and 1.5);
        \draw[black] (2,0) arc (180:0:10.5 and 1.5);
        \draw[black] (4,0) arc (180:0:10.5 and 1.5);
        \draw[black] (7,0) arc (180:0:1 and 1);
        \draw[black] (11,0) arc (180:0:1 and 1);
        \draw[black] (18,0) arc (180:0:1 and 1);
        \draw[black] (22,0) arc (180:0:2.5 and 1.5);

        \node[anchor=north west] at (-0.6, -0.05) {\scriptsize{1}};
        \node[anchor=north west] at (4.4, -0.05) {\scriptsize{$b_1$}};
        \node[anchor=north west] at (21.3, -0.1) {\scriptsize{$a_n$}};
        \node[anchor=north west] at (26.4, -0.05) {\scriptsize{$2n$}};
        
     \node[anchor=north west] at (1.6, 1.7) {\footnotesize{$A$}};
     \node[anchor=north west] at (5.3, 1.1) {\footnotesize{$B$}};
    \node[anchor=north west] at (7.3, 1.2) {\footnotesize{$C$}};
    \node[anchor=north west] at (9.1, 1.2) {\footnotesize{$Y_1$}};
    \node[anchor=north west] at (11.1, 1.2) {\footnotesize{$X_1$}};
    \node[anchor=north west] at (13, 1.2) {\footnotesize{$Y_2$}};
            \node[anchor=north west] at (14.7, 1.2) {\footnotesize{$\cdots$}};
     \node[anchor=north west] at (16.6, 1.2) {\footnotesize{$Y_k$}};
     \node[anchor=north west] at (18.2, 1.2) {\footnotesize{$X_k$}};
     \node[anchor=north west] at (20.4, 1.1) {\footnotesize{$D$}};
     \node[anchor=north west] at (23.7, 1.7) {\footnotesize{$E$}};
    \end{tikzpicture} 
\caption{The structure of $P_3$-avoiding Stoimenow matchings in Case 3.}\label{case3-P3}
\end{figure}
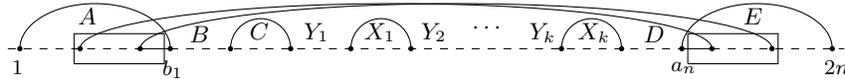 

\begin{itemize}
\item  The possibly empty run of openers $(b_1-i)(b_1-i+1)\cdots (b_1-1)$ has closers $(a_n+1)(a_n+2)\cdots (a_n+i)$ which contributes a factor of $\frac{1}{1-x}$ to the formula derived below for this case; the run of openers $(2) (3)\cdots(b_1-i-1)$ have closers between $b_1$ and $a_n$.
\item The letter under an arc \([a, b]\) represents the interval of all openers and closers between \(a\) and \(b\) (these letters are \(A\), \(C\), \(X_1, \ldots, X_k\), and \(E\)), while a letter between two consecutively drawn arcs \([a, b]\) and \([c, d]\) represents the interval of all openers and closers between \(b\) and \(c\) (these letters are \(B\), \(Y_1, \ldots, Y_k\), and \(D\)). Except for \(A\) and \(E\), the interval represented by a letter inside an arc may contain other arcs. 
\item To avoid an occurrence of $P_3$ involving arcs $[1,b_1]$ and $[a_n,2n]$, an interval $I$, different from $A$ and $E$,  under an arc has closers (resp., openers) of arcs with openers (resp., closers) in $I$ or $A$ (resp., $I$ or $E$). We let $H(x)$ be the g.f.\ for such an interval $I$ together with the arc that defines it. $H(x)$ is derived in Lemma~\ref{lem-H(x)}, and it is used below to derive $F(x)$.
\item For any interval \(Y_j\), it is either empty, or it begins with a \(\beta\), ends with an \(\alpha\), and has an arbitrary (possibly empty) binary string over \(\{\alpha, \beta\}^*\) in the positions in between. Hence, the g.f.\ for each $Y_j$ is $1+\frac{x^2}{1-2x}$.  Since each $Y_j$ comes in pair with $X_j$, the g.f.\ for the interval $Y_jx_jX_jy_j$, where $[x_j,y_j]$ is the arc defining $X_j$, is $\frac{1}{1-\left(1+\frac{x^2}{1-2x}\right)H(x)}$ (there may be no $Y_j$'s and $X_j$'s).
\item Finally, for interval \(B\) (resp., \(D\)), the only requirement is that it ends with \(\alpha\) (resp., begins with \(\beta\)), if it is not empty. The g.f.\ for these intervals together with the arcs $[1,b_1]$, $[a_n, 2n]$, $C$ and its arc is therefore \(x^2\left(1 + \frac{x}{1 - 2x}\right)^2H(x)\).
\end{itemize}

Summarizing the three cases above, we obtain
\begin{equation}\label{F(x)-P3-formula}
\footnotesize
F(x)=\frac{1}{1-x}\left(1+\frac{x^2}{1-2x}+x^2\left(1+\frac{x}{1-2x}\right)^2\frac{H(x)}{1-\left(1+\frac{x^2}{1-2x}\right)H(x)}\right).
\end{equation}\\[-3mm]

\begin{figure}
 \centering
    \begin{tikzpicture}[scale = 0.4] 
        \draw [dashed](-0.4,0) -- (22.4,0);
        \foreach \x in {0,3,6,8,11,22}
        {
            \filldraw (\x,0) circle (2pt);
        }
        
        \filldraw (18.7,0) circle (2pt);
         \filldraw (15.7,0) circle (2pt);
       
        \draw[black] (0,0) arc (180:0:11 and 2.5);
        \draw[black] (3,0) arc (180:0:1.5 and 1);
        \draw[black] (8,0) arc (180:0:1.5 and 1);
        \draw[black] (15.7,0) arc (180:0:1.5 and 1);
        
     \node[anchor=north west] at (0.3, 1.1) {\footnotesize{$K_0$}};
     \node[anchor=north west] at (2, 1) {\footnotesize{$\alpha$}};
    \node[anchor=north west] at (3.9, 1.1) {\footnotesize{$I_1$}};
    \node[anchor=north west] at (6, 1.1) {\footnotesize{$K_1$}};
    \node[anchor=north west] at (8.8, 1.1) {\footnotesize{$I_2$}};
    \node[anchor=north west] at (10.9, 1.1) {\footnotesize{$K_2$}};
        \node[anchor=north west] at (12.1, 1.2) {\footnotesize{$\cdots$}};
     \node[anchor=north west] at (13.2, 1.1) {\footnotesize{$K_{k-1}$}};
     \node[anchor=north west] at (16.4, 1.1) {\footnotesize{$I_k$}};
     \node[anchor=north west] at (18.5, 1.2) {\footnotesize{$\beta$}};
     \node[anchor=north west] at (19.5, 1.1) {\footnotesize{$K_{k}$}};
      \node[anchor=north west] at (-0.6, -0.05) {\footnotesize{$x_j$}};
     \node[anchor=north west] at (21.3, -0.05) {\footnotesize{$y_j$}};
    \end{tikzpicture} 
\caption{The structure of intervals under arcs in Lemma~\ref{lem-H(x)}.}\label{structureH(x)-P3}
\end{figure}
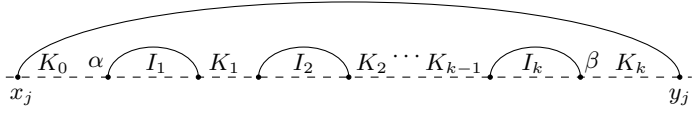

\begin{lem}\label{lem-H(x)} We have 
\begin{equation}\label{H(x)-formula}
H(x)=\frac{1-\sqrt{1-4x}}{2(1-x)}=\frac{xC(x)}{1-x}.
\end{equation}
\end{lem}

\begin{proof} As already observed above, any interval \( I \) under an arc must be formed by \( \alpha \)'s and \( \beta \)'s shuffled with arcs whose opener and closer are both contained in \( I \). If there are no such arcs in \( I \), the g.f.\ for this case is \( \frac{x}{1 - 2x} \), where the \( x \) in the numerator corresponds to the arc defining~\( I \). Figure~\ref{structureH(x)-P3} presents the situation when there are $k$, $k\geq 1$, arcs in $I$.  In this case, 
\begin{itemize}
\item \( I_1 \) must be preceded by an \( \alpha \), or else Type 1 arcs occur.  
\( I_k \) must be followed by a \( \beta \), or else Type 2 arcs occur.  
The g.f.\ for \( I_j \), together with the arc defining it, is \( H(x) \) for \( 1 \leq j \leq k \).
\item Observe that $\alpha$ (resp., $\beta$) in Figure~\ref{structureH(x)-P3} does not belong to $K_0$ (resp., $K_k$).  $K_i$'s can be any (possibly empty) binary strings over $\{\alpha,\beta\}^*$ for $0\leq i\leq k$, with additional restriction for nonempty $K_j$ to begin with a $\beta$ and to end with an $\alpha$ for $1\leq j\leq k-1$. Hence, the g.f.\ for $K_0$ and $K_k$ is $\frac{1}{1-2x}$ and for all other $K_i$'s it is $1+\frac{x^2}{1-2x}$. 
\item We can think of $I_i$ and $K_{i-1}$ be coming in pair for $2\leq i\leq k$.  The g.f.\ for a pair is $\left(1+\frac{x^2}{1-2x}\right)H(x)$. 
\end{itemize}
Summarizing all our considerations above, we conclude that  
\begin{equation}\label{equ-H3}
H(x) = \frac{x}{1 - 2x} + \frac{x^3}{1 - 2x} \cdot \frac{H(x)}{1 - 2x} \cdot \frac{1}{1 - \left(1 + \frac{x^2}{1 - 2x} \right) H(x)}
\end{equation}
where \( \frac{x^3}{1 - 2x} \) corresponds to the \( \alpha \), \( \beta \), and \( K_k \) in Figure~\ref{structureH(x)-P3} together with the arc defining \( I \),  and \( \frac{H(x)}{1 - 2x} \) corresponds to \( K_0 \) and \( I_1 \).  
Solving (\ref{equ-H3}) for $H(x)$, we obtain (\ref{H(x)-formula}), which comples the proof of Lemma~\ref{lem-H(x)}.\end{proof}

Substituting (\ref{H(x)-formula}) into (\ref{F(x)-P3-formula}), we obtain $F(x)=C(x)$ completing our proof of Theorem~\ref{MnP3Cn}.
\end{proof}

\subsection{Three length-3 Wilf-equivalent patterns}\label{sub-length-3-avoiding}
Two patterns \( P \) and \( Q \) are {\it Wilf-equivalent} if \( |\mathcal{M}_n(P)| = |\mathcal{M}_n(Q)| \) for all \( n \geq 0 \). In this subsection, we study the enumeration of Stoimenow matchings avoiding the following patterns, whose Wilf-equivalence was established bijectively in~\cite{LvKitZhang}.
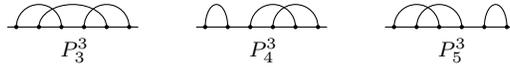
\begin{center}
\begin{tikzpicture}[scale = 0.3]
    \draw (-0.4,0) -- (5.4,0);
    \foreach \x in {0,1,2,3,4,5} \filldraw (\x,0) circle (2pt);
    \draw[black] (0,0) arc (180:0:1 and 1);
    \draw[black] (1,0) arc (180:0:1.5 and 1);
    \draw[black] (3,0) arc (180:0:1 and 1);
    \node[anchor=north west] at (1.5, -0.05) {\scriptsize{$P^3_3$}};
\end{tikzpicture}
\hspace{0.5cm}
\begin{tikzpicture}[scale = 0.3]
    \draw (-0.4,0) -- (5.4,0);
    \foreach \x in {0,1,2,3,4,5} \filldraw (\x,0) circle (2pt);
    \draw[black] (0,0) arc (180:0:0.5 and 1);
    \draw[black] (2,0) arc (180:0:1 and 1);
    \draw[black] (3,0) arc (180:0:1 and 1);
    \node[anchor=north west] at (1.5, -0.05) {\scriptsize{$P^3_4$}};
\end{tikzpicture}
\hspace{0.5cm}
\begin{tikzpicture}[scale = 0.3]
    \draw (-0.4,0) -- (5.4,0);
    \foreach \x in {0,1,2,3,4,5} \filldraw (\x,0) circle (2pt);
    \draw[black] (0,0) arc (180:0:1 and 1);
    \draw[black] (1,0) arc (180:0:1 and 1);
    \draw[black] (4,0) arc (180:0:0.5 and 1);
    \node[anchor=north west] at (1.5, -0.05) {\scriptsize{$P^3_5$}};
\end{tikzpicture}
\end{center}
\vspace{-0.5cm}
The results of the next theorem will be applied in Section~\ref{pairs-sec}.

\begin{thm}\label{thm-conj-case-k=3} 
 We have $a_n=|\M_n(P^3_3)|=
    |\M_n(P^3_4)|=
    |\M_n(P^3_5)|= 2^{n-1}$ for $n\geq 1$ with $a_0=1$, and the g.f.\ is given by
    \begin{align}\label{gf-length-3}
    \sum_{n\geq 0} a_n\, x^n = \frac{1-x}{1-2x}.
    \end{align}
 \end{thm}

\begin{proof} Since \(|\mathcal{M}_n(P^3_4)| = |\mathcal{M}_n(P^3_5)|\), we only need to consider $P^3_2$ and $P^3_4$. \\[-3mm]

\noindent
{\bf Pattern $P^3_2$.} Note that if \( M \in \mathcal{M}_n(P^3_2) \) is irreducible, then \( M \) must be an \( n \)-crossing, so there is only one such matching. On the other hand, if \( M \) is reducible, the first (leftmost) irreducible block must be a \( k \)-crossing for some \( k \) (only one choice). The remaining part of the matching can be any \( P^3_2 \)-avoiding matching. Hence, if we let \( a_n = |\mathcal{M}_n(P^3_2)| \), we have $a_n = 1 + \sum_{k=1}^{n-1} a_{n-k}$, which leads to \( a_n = 2^{n-1} \), since \( a_0 = 1 \). \\[-3mm]

\noindent
{\bf Pattern $P^3_4$.} For $M\in\M_n(P^3_4)$, the matching induced by the arcs in $X=\{[a,b]|a>b_1\}$, along with the arc $[1,b_1]$, must be a $k$-noncrossing for some $k$. For an arc with an opener under $[1,b_1]$, we label its closer by $\alpha$ (all these arcs induce an $(n-k)$-crossing), and we label each arc in $X$ by $\beta$. 

Next, we describe a bijection between matchings in \( \mathcal{M}_n(P^3_4) \) and binary strings of length \( n - 1 \) over the alphabet \( \{ \alpha, \beta \} \), which will complete the proof of the theorem.  
Given such a string \( x_1x_2\cdots x_{n-1} \), the \( \beta \)'s in it (if any), from left to right, correspond to the arcs in \( X \) in increasing order of their openers.

If \( x_i = \alpha \), then place a closer inside the arc marked by the \( \beta \) closest to \( x_i \) on the left. There will be a unique position for the closer, as it must be part of the \((n-k)\)-crossing.  
On the other hand, if no such \( \beta \) exists, place the closer in the unique (smallest) position immediately to the right of the arc \([1, b_1]\). For example, the string $\beta\beta\alpha\beta\alpha\alpha\beta$ corresponds to the matching shown in the left figure below, while the string $\alpha\alpha\beta\alpha\beta$ corresponds to the matching shown in the right figure below.
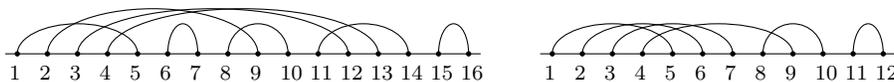
\begin{center}
    \begin{tikzpicture}[scale = 0.4] 
        \draw (-0.4,0) -- (15.4,0);
        \foreach \x in {0,1,2,3,4,5,6,7,8,9,10,11,12,13,14,15}
        {
            \filldraw (\x,0) circle (2pt);
        }
        
        \draw[black] (0,0) arc (180:0:2 and 1);
        \draw[black] (1,0) arc (180:0:3.5 and 1.5);
        \draw[black] (2,0) arc (180:0:4.5 and 1.5);
        \draw[black] (3,0) arc (180:0:4.5 and 1.5);
        \draw[black] (5,0) arc (180:0:0.5 and 1);
        \draw[black] (7,0) arc (180:0:1 and 1);
        \draw[black] (10,0) arc (180:0:1.5 and 1);
        \draw[black] (14,0) arc (180:0:0.5 and 1);

         \node[anchor=north west] at (-0.6, -0.05) {\scriptsize{1}};
         \node[anchor=north west] at (0.4, -0.05) {\scriptsize{2}};
         \node[anchor=north west] at (1.4, -0.05) {\scriptsize{3}};
          \node[anchor=north west] at (2.4, -0.05) {\scriptsize{4}};
         \node[anchor=north west] at (3.4, -0.05) {\scriptsize{5}};
         \node[anchor=north west] at (4.4, -0.05) {\scriptsize{6}};
         \node[anchor=north west] at (5.4, -0.05) {\scriptsize{7}};
         \node[anchor=north west] at (6.4, -0.05) {\scriptsize{8}};
         \node[anchor=north west] at (7.4, -0.05) {\scriptsize{9}};
          \node[anchor=north west] at (8.4, -0.05) {\scriptsize{10}};
         \node[anchor=north west] at (9.4, -0.05) {\scriptsize{11}};
         \node[anchor=north west] at (10.4, -0.05) {\scriptsize{12}};
         \node[anchor=north west] at (11.4, -0.05) {\scriptsize{13}};
          \node[anchor=north west] at (12.4, -0.05) {\scriptsize{14}};
         \node[anchor=north west] at (13.4, -0.05) {\scriptsize{15}};
         \node[anchor=north west] at (14.4, -0.05) {\scriptsize{16}};
    \end{tikzpicture}
    \hspace{0.3cm}
    \begin{tikzpicture}[scale = 0.4] 
        \draw (-0.4,0) -- (11.4,0);
        \foreach \x in {0,1,2,3,4,5,6,7,8,9,10,11}
        {
            \filldraw (\x,0) circle (2pt);
        }
        
        \draw[black] (0,0) arc (180:0:2 and 1);
        \draw[black] (1,0) arc (180:0:2 and 1);
        \draw[black] (2,0) arc (180:0:2 and 1);
        \draw[black] (3,0) arc (180:0:2.5 and 1);
        \draw[black] (7,0) arc (180:0:1 and 1);
        \draw[black] (10,0) arc (180:0:0.5 and 1);

        \node[anchor=north west] at (-0.6, -0.05) {\scriptsize{1}};
         \node[anchor=north west] at (0.4, -0.05) {\scriptsize{2}};
         \node[anchor=north west] at (1.4, -0.05) {\scriptsize{3}};
          \node[anchor=north west] at (2.4, -0.05) {\scriptsize{4}};
         \node[anchor=north west] at (3.4, -0.05) {\scriptsize{5}};
         \node[anchor=north west] at (4.4, -0.05) {\scriptsize{6}};
         \node[anchor=north west] at (5.4, -0.05) {\scriptsize{7}};
         \node[anchor=north west] at (6.4, -0.05) {\scriptsize{8}};
         \node[anchor=north west] at (7.4, -0.05) {\scriptsize{9}};
          \node[anchor=north west] at (8.4, -0.05) {\scriptsize{10}};
         \node[anchor=north west] at (9.4, -0.05) {\scriptsize{11}};
         \node[anchor=north west] at (10.4, -0.05) {\scriptsize{12}};
    \end{tikzpicture}
\end{center}
\vspace{-0.2cm}
It is easy to see that the map described by us is a bijection.

Thus, we obtain $a_n$ and can easily derive~\eqref{gf-length-3}. The proof is complete.
\end{proof}

\section{Avoidance of pairs of patterns}\label{pairs-sec}

\subsection{Pair $(P_1, P_2)$}

In order to present our results for the pair \((P_1, P_2)\), we need the notion of a \textbf{(2+2)}-free poset and the description of the bijection \(\Omega\) from the set of Stoimenow matchings to \textbf{(2+2)}-free posets, introduced in~\cite{Bousquet-Claesson-Dukes-Kitaev-DukesParvi}. A poset \(P = (X, \preceq)\) is \textbf{(2+2)}-{\it free} if it does not contain an induced subposet isomorphic to \textbf{(2+2)}, the union of two disjoint 2-element chains. We let \(\mathcal{P}_n\) denote the set of \textbf{(2+2)}-free posets with \(n\) elements, and \(\mathcal{P}_n(X)\) denote the set of \textbf{(2+2)}-free posets of length \(n\) that avoid a given set of patterns \(X\). 

The map \(\Omega\) sends arcs in \(M \in \mathcal{M}_n\) to elements of a poset \(P = \Omega(M) \in \mathcal{P}_n\) such that for any two arcs \([a_i, b_i]\) and \([a_j, b_j]\), we have \(b_i < a_j\) if and only if \(\Omega([a_i, b_i]) \prec \Omega([a_j, b_j])\).

A poset avoids \textbf{(3+1)} if it does not contain an induced subposet isomorphic to the disjoint union of a 3-element chain and a 1-element chain. The poset \textbf{N} is the poset on the elements in $\{a,b,c,d\}$ with the relations $a < c$, $a < d$, and $b < d$. A poset is \textbf{N}-avoiding if it does not contain an induced subposet isomorphic to \textbf{N}.

\begin{thm}\label{thm-P1-P2}
    The g.f.\ for $\M(P_1,P_2)$ is given by 
     \begin{align}\label{A(x)-P_1-P_2}
        A(x)=\frac{(1-x)^3}{1-4x+5x^2-3x^3}.
     \end{align}
The corresponding sequence is {\rm A116703} in {\rm \cite{OEIS}}.
    \end{thm}
    
 \begin{proof} 
From Theorem 4.2 (resp., Theorem 4.3) in~\cite{LvKitZhang}, it follows that \(\Omega\) maps bijectively \(\mathcal{M}_n(P_1)\) to \(\mathcal{P}_n(\text{\textbf{3+1}})\) (resp., \(\mathcal{M}_n(P_2)\) to \(\mathcal{P}_n(\textbf{N})\)). Thus, \(|\mathcal{M}_n(P_1, P_2)| = |\mathcal{P}_n(\text{\textbf{3+1}}, \textbf{N})|\) for all \(n \geq 0\), and (\ref{A(x)-P_1-P_2}) follows from Section 3.4 in~\cite{DisantoPerPinRin}; our g.f.\ is that in~\cite{DisantoPerPinRin}, plus 1 to account for the empty matching.
\end{proof}

\subsection{Pair $(P_1, P_3)$}\label{subsec-P1-P3}

\begin{thm}\label{thm-P1-P3}
    The g.f.\ for $\M(P_1,P_3)$ is given by 
    \begin{align}\label{A(x)-P_1-P_3}
        A(x)=\frac{1-6x+15x^2-19x^3+13x^4-5x^5}{(1-x)^5(1-2x)}.
    \end{align}
    \end{thm}
\begin{proof}
Given a matching \( M \in \mathcal{M}_n(P_1, P_3) \), the first two cases to consider for \( M \), namely, when \( M \) is an \( n \)-crossing and when there are no arcs between \( b_1 \) and \( a_n \), are the same as those in our analysis of \( \mathcal{M}(P_3) \) in Section~\ref{subsec-P3-avoid}. Together, these cases contribute $\frac{1}{1 - x} \left( 1 + \frac{x^2}{1 - 2x} \right)$ to \( A(x) \). The remaining two cases are as follows.

\begin{figure}
 \centering
    \begin{tikzpicture}[scale = 0.4] 
        \draw (-0.4,0) -- (1,0);
        \draw [dashed](1,0) -- (3,0);
         \draw (3,0) -- (4,0);
         \draw [dashed](4,0) -- (6,0);
         \draw (6,0) -- (7,0);
         \draw [dashed](7,0) -- (9,0);
         \draw (9,0) -- (10,0);
         \draw[dashed] (10,0) -- (12,0);
         \draw (12,0) -- (13,0);
         \draw[dashed] (13,0) -- (15,0);
         \draw (15,0) -- (16,0);
         \draw[dashed] (16,0) -- (18,0);
         \draw (18,0) -- (19.4,0);
        \foreach \x in {0,1,3,4,6,7,9,10,12,13,15,16,18,19}
        {
            \filldraw (\x,0) circle (2pt);
        }
    \draw[domain=0:90, smooth, variable=\t, shift={(2,0)}] 
    plot ({-cos(\t)}, {sin(\t)});
    \draw[domain=0:90, smooth, variable=\t, shift={(14,0)}] 
    plot ({-cos(\t)}, {sin(\t)});
    \draw[domain=0:90, smooth, variable=\t, shift={(5,0)}] 
    plot ({cos(\t)}, {sin(\t)});
    \draw[domain=0:90, smooth, variable=\t, shift={(17,0)}] 
    plot ({cos(\t)}, {sin(\t)});
  
        \draw[black] (0,0) arc (180:0:2 and 1.5);
        \draw[black] (3,0) arc (180:0:3.5 and 1.5);
        \draw[black] (7,0) arc (180:0:2.5 and 1.5);
        \draw[black] (9,0) arc (180:0:3.5 and 1.5);
        \draw[black] (15,0) arc (180:0:2 and 1.5);

        \node[anchor=north west] at (-0.7, -0.05) {\scriptsize{1}};
        \node[anchor=north west] at (3.2, -0.05) {\scriptsize{$b_1$}};
        \node[anchor=north west] at (6.4, -0.05) {\scriptsize{$x$}};
        \node[anchor=north west] at (11.4, -0.05) {\scriptsize{$y$}};
        \node[anchor=north west] at (14.2, -0.05) {\scriptsize{$a_n$}};
        \node[anchor=north west] at (18.2, -0.05) {\scriptsize{$2n$}};
    \end{tikzpicture}
\caption{The structure of $(P_1,P_3)$-avoiding Stoimenow matchings in Case~3.}\label{fig-pair-P1-P3}
\end{figure}
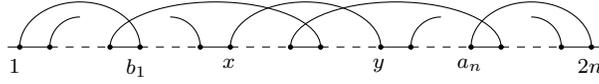 

\noindent    
\textbf{Case 3.} If there is only one arc, denoted \([x, y]\), between \( b_1 \) and \( a_n \), as shown in Figure~\ref{fig-pair-P1-P3}, then the arcs with closers \( b_2, b_3, \ldots, x-1 \) have openers that lie under the arc \([1, b_1]\). Similarly, the arcs with openers \( y+1, \ldots, a_{n-2}, a_{n-1} \) have closers that lie under the arc \([a_n, 2n]\). All these arcs are counted by $\left( \frac{x}{1 - x} \right)^2$. Moreover, there may be closers of arcs whose openers lie under the arc \([1, b_1]\), and openers of arcs whose closers lie under the arc \([a_n, 2n]\), which can be shuffled arbitrarily under the arc \([x, y]\) without creating a forbidden configuration. These are counted by $\frac{x}{1 - 2x}$.

\noindent
\textbf{Case 4. } Suppose there are at least two arcs between \( b_1 \) and \( a_n \). Using Figure~\ref{case3-P3}, which relates to the \( \mathcal{M}(P_3) \) analysis in Case 4 of Section~\ref{subsec-P3-avoid}. In this case, $B$ (resp., $D$) can only contain closers (resp., openers) whose corresponding openers (resp., closers) lie under the arc $[1,b_1]$ (resp., $[a_n, 2n]$). Such closers and openers, together with the arcs $[1,b_1]$ and $[a_n, 2n]$, are counted by $\left(\frac{x}{1-x}\right)^2$.  
Moreover, there must exist intervals $X_1$ and $X_k$ that can only contain closers with arcs whose openers lie under the arc $[1,b_1]$ and openers with arcs whose closers lie under the arc $[a_n,2n]$, which gives $\left(\frac{x}{1-x}\right)^2$. The remaining possibly empty intervals $X_2,\dots, X_{k-1}$, which do not intersect with any other arcs, contribute $\frac{1}{1-x}$. Note that $Y_i$'s, $i\in \{1,\ldots,k\}$, must be empty to avoid the pattern $P_1$. Thus, this case totally gives $(\frac{x}{1-x})^2 (\frac{x}{1-x})^2 \frac{1}{1-x}=\frac{x^4}{(1-x)^5}$.

    By summing over the above four cases, we obtain
    \begin{small}
    \begin{align}\label{A(x)-P1-P3}
    A(x)=\frac{1}{1-x}\left(1+\frac{x^2}{1-2x}\right)+\left(\frac{x}{1-x}\right)^2\frac{x}{1-2x}+\frac{x^4}{(1-x)^5}.
   \end{align}
   \end{small}
   Simplifying this expression, we derive~\eqref{A(x)-P_1-P_3}.
\end{proof}

\subsection{Pairs $(P_1, P_4)$, $(P_1, P_5)$ and $(P_2, P_3)$}\label{subsec-P1-P4}

\begin{thm}\label{thm-A005183}
    The g.f.\ for $\M(P_1,P_4)$, $\M(P_1, P_5)$ and $\M(P_2,P_3)$ is given by 
    \begin{align}\label{A(x)-P_1-P_4}
        A(x)=\frac{1-4x+5x^2-x^3}{(1-x)(1-2x)^2}.
    \end{align}
    The corresponding sequence is {\rm A005183} in {\rm \cite{OEIS}}.
    \end{thm}
    
\begin{figure}[t]
 \centering
    \begin{tikzpicture}[scale = 0.4] 
        \draw [dashed](-0.6,0) -- (12,0);
        \draw (12,0) -- (16,0);
        \draw [dashed](16,0) -- (18,0);
        \draw (18,0) -- (21,0);
        \draw (21,0) -- (23,0);
        \draw [dashed](23,0) -- (25,0);
        \draw (25,0) -- (28,0);
        \draw [dashed](28,0) -- (29,0);
        \foreach \x in {0,3,5,6,7,8,10,11,12,13,14,15,16,18,19,20,21,22,23,25,26,27,28}
        {
            \filldraw (\x,0) circle (2pt);
        }

\draw[domain=0:90, smooth, variable=\t, shift={(5,0)}] 
    plot ({cos(\t)}, {sin(\t)});
\draw[domain=0:90, smooth, variable=\t, shift={(7,0)}] 
    plot ({cos(\t)}, {sin(\t)});
    \draw[domain=0:90, smooth, variable=\t, shift={(10,0)}] 
    plot ({cos(\t)}, {sin(\t)});
        
        \draw[black] (0,0) arc (180:0:1.5 and 1.5);
        \draw[black] (5,0) arc (180:0:3.5 and 1.5);
        \draw[black] (7,0) arc (180:0:3 and 1.5);
        \draw[black] (10,0) arc (180:0:2 and 1.5);
        \draw[black] (15,0) arc (180:0:2 and 1.5);
        \draw[black] (16,0) arc (180:0:2 and 1.5);
        \draw[black] (18,0) arc (180:0:1.5 and 1.5);
        \draw[black] (22,0) arc (180:0:2 and 1.5);
        \draw[black] (23,0) arc (180:0:2 and 1.5);
        \draw[black] (25,0) arc (180:0:1.5 and 1.5);
        \node[anchor=north west] at (8.8, 2.8) {\footnotesize{$I_1$}};
        \node[anchor=north west] at (17.2, 2.8) {\footnotesize{$I_2$}};
        \node[anchor=north west] at (24, 2.8) {\footnotesize{$I_3$}};
          \node[anchor=north west] at (3, 1.2) {\footnotesize{$\cdots$}};
     \node[anchor=north west] at (8, 1.2) {\footnotesize{$\cdots$}};
     \node[anchor=north west] at (28, 1.2) {\footnotesize{$\cdots$}};
        \node[anchor=north west] at (-0.6, -0.1) {\scriptsize{1}};
        \node[anchor=north west] at (2.4, -0.1) {\scriptsize{$b_1$}};
        \node[anchor=north west] at (4.4, -0.1) {\scriptsize{$x_1$}};
        \node[anchor=north west] at (6.4, -0.1) {\scriptsize{$x_2$}};
        \node[anchor=north west] at (9.3, -0.1) {\scriptsize{$x_k$}};
        \node[anchor=north west] at (11.4, -0.1) {\scriptsize{$y_1$}};
        \node[anchor=north west] at (12.4, -0.1) {\scriptsize{$y_2$}};
        \node[anchor=north west] at (13.4, -0.1) {\scriptsize{$y_k$}};

    \end{tikzpicture}
\caption{The structure of $(P_1, P_4)$-avoiding Stoimenow matchings.}\label{fig-P_1-P_4}
\end{figure}
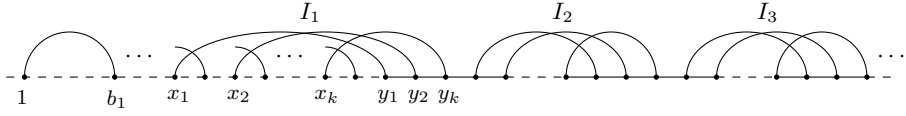
    
\begin{proof}
    \textbf{Pairs $(P_1,P_4)$ and $(P_1, P_5)$.}
 Using the reverse operation, we only need to consider the pair $(P_1, P_4)$. Given \( M \in \mathcal{M}_n(P_1, P_4) \), if \( M \) is an \( n \)-crossing, where $n \geq 1$, then its contribution to \( A(x) \) is \( \frac{1}{1 - x} \). If $M$ is not an $n$-crossing, we let The block \(I_i\) is composed of a maximal \(k\)-crossing (\(k \geq 1\)) formed by arcs \([x_1, y_1], [x_2, y_2], \ldots, [x_k, y_k]\) where \(x_1 > b_1\) and \(x_1 < \cdots < x_k < y_1 < \cdots < y_k\), combined with additional arcs \([a, b]\) when \(x_i < a < b < x_{i+1}\) holds for some \(1 \leq i \leq k - 1\), along with closers positioned within the intervals \((x_i, x_{i+1})\) and \((x_k, y_1)\). Then we have $I_2$, $I_3$, etc, is an $|I_j|$-crossing, $j=2,3,\ldots$, as shown in Figure~\ref{fig-P_1-P_4}. 
 This contributes a factor of $\sum_{j \geq 0} \left(\frac{x}{1-x}\right)^j$ in (\ref{aux-eq-1}) below. On the other hand, the block $I_1$ (that must exist), formed by  the intervals $(x_1, x_2)$, $(x_2,x_3), \ldots, (x_k,y_1)$,  possibly contains closers with arcs whose openers lie under the arc $[1,b_1]$, which give the contribution of $\sum_{k\geq 1}x^k \left(\frac{1}{1-x}\right)^k$ to the total number of matchings in the case when $M$ is not an $n$-crossing:     
 \begin{equation}\label{aux-eq-1}
    \frac{x}{1-x}\,\sum_{j \geq 0} \left(\frac{x}{1-x}\right)^j \sum_{k\geq 1}x^k\left(\frac{1}{1-x}\right)^k= \frac{x}{(1-x)^2(1-\frac{x}{1-x})^2},
    \end{equation}
    where the factor of $\frac{x}{1-x}$ counts the arcs $[1,b_1]$, $[2,b_1+1],\ldots,[m,b_1+m-1]$, with maximum possible $m$. Thus, by summing the above two cases, we obtain 
\begin{align*}
    A(x)=\frac{1}{1-x}+\frac{x}{(1-x)^2(1-\frac{x}{1-x})^2},
\end{align*}
which coincides with~\eqref{A(x)-P_1-P_4}.\\[-3mm]

\noindent
\textbf{Pair $(P_2, P_3)$.}
Given a matching $M \in \mathcal{M}(P_2, P_3)$, we use the splitting procedure in the analysis of $\mathcal{M}(P_2)$ in Section~\ref{subsec-P2-avoiding}. This procedure is also well-defined in $\mathcal{M}(P_2, P_3)$, since the method of adding a reduction arc in the glue rule does not introduce any forbidden patterns. 
Suppose $M$ can be split into $M_1$ and $M_2$.
Let $B(M)$ be the number of irreducible blocks in $M$ and $\MM$ be the matching formed by its first irreducible block. We have the following three cases.
\begin{itemize}
    \item If $B(M) > 1$ and $\text{redarc}(\MM) \neq [1,b_1]$, then $M_1, M_2\neq\emptyset$. Moreover, we must have $M_1\in \M(P_2, 
\begin{tikzpicture}[scale = 0.25]
    \draw (-0.4,0) -- (5.4,0);
    \foreach \x in {0,1,2,3,4,5}
    {
        \filldraw (\x,0) circle (2pt);
    }
    \draw[black] (0,0) arc (180:0:0.5 and 1);
    \draw[black] (2,0) arc (180:0:1 and 1);
    \draw[black] (3,0) arc (180:0:1 and 1);
\end{tikzpicture})$
 and $M_2\in\M(P_2, 
\begin{tikzpicture}[scale = 0.25]
    \draw (-0.4,0) -- (5.4,0);
    \foreach \x in {0,1,2,3,4,5}
    {
        \filldraw (\x,0) circle (2pt);
    }
    \draw[black] (0,0) arc (180:0:1 and 1);
    \draw[black] (1,0) arc (180:0:1 and 1);
    \draw[black] (4,0) arc (180:0:0.5 and 1);
\end{tikzpicture})$. This means $M_1,M_2\in\M(\begin{tikzpicture}[scale = 0.25]
    \draw (-0.4,0) -- (5.4,0);
    \foreach \x in {0,1,2,3,4,5}
    {
        \filldraw (\x,0) circle (2pt);
    }
    \draw[black] (0,0) arc (180:0:0.5 and 1);
    \draw[black] (2,0) arc (180:0:1 and 1);
    \draw[black] (3,0) arc (180:0:1 and 1);
\end{tikzpicture})$, each counted by $\frac{x}{1-2x}$ via Theorem~\ref{thm-conj-case-k=3}. Thus, the reduction arc of $\MM$, together with $M_1$ and $M_2$, contributes a factor of $x\left(\frac{x}{1 - 2x}\right)^2$ to $A(x)$. 
\item If $B(M)=1$, then $M_2=\emptyset$, and therefore $M_1 \in \M_{n-1}(P_2,P_3)$, which gives $xA(x)$.
\item If $B(M)>1$ and redarc$(\MM)=[1,b_1]$, then $M_1=\emptyset$, $M_2\neq\emptyset$, and $M_2\in\M(P_2,
\begin{tikzpicture}[scale = 0.25]
    \draw (-0.4,0) -- (5.4,0);
    \foreach \x in {0,1,2,3,4,5}
    {
        \filldraw (\x,0) circle (2pt);
    }
    \draw[black] (0,0) arc (180:0:0.5 and 1);
    \draw[black] (2,0) arc (180:0:1 and 1);
    \draw[black] (3,0) arc (180:0:1 and 1);
\end{tikzpicture})$.
The arc $[1,b_1]$, together with $M_2$, is counted by $\frac{x^2}{1-2x}$.
\end{itemize}
Summarizing the above three cases, we obtain  
\[
A(x) = x\left(\frac{x}{1 - 2x}\right)^2 + xA(x) + \frac{x^2}{1 - 2x}.
\]  
By solving this equation, we derive~\eqref{A(x)-P_1-P_4}. The proof is complete.
\end{proof}

\subsection{Pairs $(P_2,P_4)$, $(P_2, P_5)$, $(P_3,P_4)$, $(P_3, P_5)$ and $(P_4, P_5)$}\label{subsec-A001519}

\begin{thm}\label{thm-A001519}
The g.f.\ for each of the pairs $\mathcal{M}(P_2, P_4)$, $\mathcal{M}(P_2, P_5)$, $\mathcal{M}(P_3, P_4)$, $\mathcal{M}(P_3, P_5)$, and $\mathcal{M}(P_4, P_5)$ is given by
\begin{align} \label{gf-P_2-P_4}
    A(x) = \frac{1 - 2x}{1 - 3x + x^2}.
\end{align}
Moreover, we have $|\mathcal{M}_n| = F_{2n-1}$, where $F_{2n-1}$ denotes the $(2n - 1)$-st Fibonacci number. The corresponding sequence is {\rm A001519} in {\rm \cite{OEIS}}.
\end{thm}
\begin{proof}
 \textbf{Pairs $(P_2,P_4)$ and $(P_2, P_5)$.}
  Since reversing all matchings shows that $|\M_n(P_2,P_4)|=|\M_n(P_2, P_5)|$, we only need to consider the pair $(P_2,P_4)$. 

 For any $\emptyset \neq M \in \mathcal{M}_n(P_2, P_4)$, we use the {\it decomposition of the first irreducible block}. That is, let $M = (M_1, M_2)$, where $M_1$ represents the matching formed by the first irreducible block of $M$, and $M_2$ represents the matching formed by the remaining irreducible blocks of $M$. We have the following two cases:
 \begin{itemize}
     \item If $M_2 \neq \emptyset$, then $M_1$ must be an $|M_1|$-crossing to avoid the pattern $P_4$, and such $M_1$'s are counted by $\frac{x}{1 - x}$. Meanwhile, $M_2$ is a nonempty $(P_2, P_4)$-avoiding Stoimenow matching, and such $M_2$'s are counted by $A(x) - 1$. 
     \item If $M_2 = \emptyset$, then there is again a possiblly empty $(P_2, P_4)$-avoiding Stoimenow matching after removing the reduction arc. Thus, this case contributes a factor of $xA(x)$ to $A(x)$.
 \end{itemize}
Combining the cases above, we obtain
\[
A(x) = 1 + \frac{x}{1-x} (A(x) - 1) + xA(x),
\]
where 1 represents the empty matching. Solving this equation for \(A(x)\), we derive \(\eqref{gf-P_2-P_4}\).\\[-3mm]

\noindent
\textbf{Pairs $(P_3,P_4)$ and $(P_3, P_5)$.}
Because of the reverse operation, it suffices to consider \(\mathcal{M}(P_3, P_4)\). 

For $M \in \mathcal{M}_n(P_3, P_4)$, using the observations in Section~\ref{subsec-P3-avoid} related to the pattern $P_3$, the only difference arises in Case~3. Namely, in comparison with Figure~\ref{case3-P3}, there are no arcs, or openers of arcs whose closers lie within the arc $[a_n, 2n]$, inside the intervals $C, X_1, \ldots, X_k$ to avoid the pattern $P_4$. 
Therefore, each interval, along with the closers in it, is counted by $\frac{x}{1 - x}$. 

By replacing $H(x)$ with $\frac{x}{1 - x}$ in~\eqref{F(x)-P3-formula}, we obtain
\begin{align*}
  \scriptsize
  A(x) = \frac{1}{1 - x} \left( 1 + \frac{x^2}{1 - 2x} + x^2 \left( 1 + \frac{x}{1 - 2x} \right)^2  \frac{\frac{x}{1 - x}}{1 - \left(1 + \frac{x^2}{1 - 2x} \right) \frac{x}{1 - x}} \right).
\end{align*}
 Simplifying this expression, we obtain~\eqref{gf-P_2-P_4}.\\[-3mm]

\noindent
\textbf{Pair $(P_4, P_5)$.}
Given \( M \in \mathcal{M}_n(P_4, P_5) \), if \( M \) is an \( n \)-crossing, where \( n \geq 0 \), then its contribution to \( A(x) \) is \( \frac{1}{1 - x} \).

\begin{figure}[h]
 \centering
    \begin{tikzpicture}[scale = 0.4] 
        \draw [dashed](-0.6,0) -- (16,0);
        \draw (16,0) -- (18.6,0);
      
        \foreach \x in {0,3,9,10,11,12,14,15,16,17,18}
        {
            \filldraw (\x,0) circle (2pt);
        }

\draw[domain=0:90, smooth, variable=\t, shift={(9,0)}] 
    plot ({cos(\t)}, {sin(\t)});
\draw[domain=0:90, smooth, variable=\t, shift={(11,0)}] 
    plot ({cos(\t)}, {sin(\t)});
    \draw[domain=0:90, smooth, variable=\t, shift={(14,0)}] 
    plot ({cos(\t)}, {sin(\t)});
        
        \draw[black] (0,0) arc (180:0:1.5 and 1.5);
        \draw[black] (9,0) arc (180:0:3.5 and 1.5);
        \draw[black] (11,0) arc (180:0:3 and 1.5);
        \draw[black] (14,0) arc (180:0:2 and 1.5);
     
        \node[anchor=north west] at (12.3, 2.8) {\footnotesize{$I'$}};
         \node[anchor=north west] at (5, 1.2) {\footnotesize{$\cdots$}};
     \node[anchor=north west] at (12, 1.2) {\footnotesize{$\cdots$}};
        \node[anchor=north west] at (-0.6, -0.1) {\scriptsize{1}};
        \node[anchor=north west] at (2.4, -0.1) {\scriptsize{$b_1$}};
        \node[anchor=north west] at (8.4, -0.1) {\scriptsize{$x_1$}};
        \node[anchor=north west] at (10.4, -0.1) {\scriptsize{$x_2$}};
        \node[anchor=north west] at (13.3, -0.1) {\scriptsize{$a_n$}};
        \node[anchor=north west] at (15.4, -0.1) {\scriptsize{$y_1$}};
        \node[anchor=north west] at (16.4, -0.1) {\scriptsize{$y_2$}};
        \node[anchor=north west] at (17.4, -0.1) {\scriptsize{$2_n$}};

    \end{tikzpicture}
\caption{The structure of $(P_4, P_5)$-avoiding Stoimenow matchings.}\label{fig-I_i}
\end{figure}
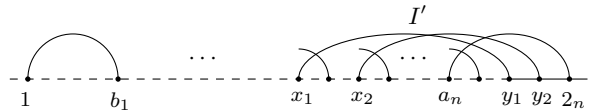

If \( M \) is not an \( n \)-crossing matching, the block \( I' \) consists of arcs $[x_1, y_1], \ldots,$  $[x_k, y_k] = [a_n, 2n]$ forming a maximal \( k \)-crossing (\( k \geq 1 \)), along with (i) additional arcs (if any) of the form \([a, b]\) satisfying \( x_i < a < b < x_{i+1} \) for \( 1 \leq i \leq k - 1 \), and (ii) closers within intervals \((x_i, x_{i+1})\) (\( 1 \leq i \leq k - 1 \)) and \((a_n, y_1)\), as shown in Figure~\ref{fig-I_i}, while \( I'' \) denotes the remaining arcs. We have the following observations.
\begin{itemize}
    \item We have $\emptyset \neq I'' \in \mathcal{M}(\tikz[scale = 0.25]{
    \draw (-0.4,0) -- (5.4,0);
    \foreach \x in {0,1,2,3,4,5}
    {
        \filldraw (\x,0) circle (2pt);
    }
    \draw[black] (0,0) arc (180:0:1 and 1);
    \draw[black] (1,0) arc (180:0:1.5 and 1);
    \draw[black] (3,0) arc (180:0:1 and 1);
})$ to avoid the pattern $P_5$, and the g.f.\ for $I''$ is $\frac{x}{1-2x}$ by Theorem~\ref{thm-conj-case-k=3}.
\item For arcs and closers in $I'$, there are possibly arcs inside the intervals $(x_i, x_{i+1})$, $i \in \{1, \dots, k-1\}$. To avoid the patterns $P_4$ and  $P_5$, each interval, along with the arcs with openers inside the arc $[1, b_1]$, induce a matching in $\mathcal{M}(\tikz[scale = 0.25]{
    \draw (-0.4,0) -- (5.4,0);
    \foreach \x in {0,1,2,3,4,5}
    {
        \filldraw (\x,0) circle (2pt);
    }
    \draw[black] (0,0) arc (180:0:1 and 1);
    \draw[black] (1,0) arc (180:0:1.5 and 1);
    \draw[black] (3,0) arc (180:0:1 and 1);
})$. Thus, the number of ways to generate each interval $(x_i, x_{i+1})$ is given by $\frac{1-x}{1-2x}$, which also allows such a interval to be empty. Moreover, in the interval $(x_k,y_1)$, there can only possibly be closers of arcs with openers inside $[1, b_1]$, which gives $\frac{1}{1-x}$. Therefore,  the contribution of $I'$ to $A(x)$ is 
\begin{small}
\begin{align*}
\frac{1}{1-x} \sum_{k \geq 1} x^k \cdot \left(1 + \frac{x}{1-2x}\right)^{k-1} = \frac{x}{(1-x)\left(1 - x\left(1 + \frac{x}{1-2x}\right)\right)}.
\end{align*}
\end{small}
\end{itemize}
Thus, we obtain
\[
A(x)=\frac{1}{1-x}+\frac{x}{1-2x}\cdot\frac{x}{(1-x)(1-x(1+\frac{x}{1-2x}))},
\]
which coincides with~\eqref{A(x)-P_1-P_4} after simplifications. 

Furthermore, the g.f.\ $A(x)$ satisfies the recurrence relation
\begin{align*}
    A(x) - x - 1 = 3x\,\left(A(x) - 1 \right) - x^2A(x).
\end{align*}
By examining the coefficients, we obtain
\begin{align*}
    a_n = 3a_{n-1} - a_{n-2},
\end{align*}
with $a_0 = a_1 = 1$, where $a_n = |\mathcal{M}_n(P_2, P_4)| = |\mathcal{M}_n(P_2, P_5)| = |\mathcal{M}_n(P_3, P_4)| = |\mathcal{M}_n(P_3, P_5)| = |\mathcal{M}_n(P_4, P_5)|$.
Since the Fibonacci number $F_n$ satisfies $F_n = F_{n-1} + F_{n-2}$ with $F_0 = F_1 = 1$, we verify that \( a_n = F_{2n-1} \). This completes the proof.\end{proof}

\section{Avoidance of triples of patterns}\label{triples-sec}

\subsection{Triples $(P_1, P_2, P_3)$}

\begin{thm}\label{thm-P1-P2-P3}
     The g.f.\ for $\M(P_1, P_2, P_3)$ is given by 
      \begin{align}\label{A(x)-P1-P2-P3}
          A(x)=\frac{1-4x+7x^2-5x^3+2x^4-x^5+x^6}{(1-x)^5}.
     \end{align}
     The corresponding sequence is {\rm A116722} in {\rm \cite{OEIS}}.
\end{thm}
\begin{proof}
Suppose \( M \in \mathcal{M}_n(P_1, P_2, P_3) \). We consider four cases, where Case~1 and Case~4 are the same as in the analysis of $\mathcal{M}_n(P_1,P_3)$ in Section~\ref{subsec-P1-P3}, contributing to $\frac{1}{1-x}$ and $\frac{x^4}{(1-x)^5}$, respectively, and the remaining two cases are as follows:
\begin{itemize}
    \item[(i)] In Case 2, if there are no arcs between \( b_1 \) and \( a_n \), then the closers and openers between \( b_1 \) and \( a_n \) cannot be shuffled arbitrarily without creating the pattern \( P_2 \). In this case, the contribution is counted by $\frac{1}{1-x} \left( \frac{x}{1-x} \right)^2 = \frac{x^2}{(1-x)^3}$.
    
    \item[(ii)] In Case 3, if there is only one arc, \([x, y]\), between \( b_1 \) and \( a_n \), then only closers or openers of arcs, whose corresponding openers or closers lie inside the intervals \([1, b_1]\) or \([a_n, 2n]\), may appear in the interval \( (x, y) \). This contributes $x \left( \frac{x}{1-x} \right)^2 \left( 1 + \frac{2x}{1-x} \right)$ in total, where the factor \( \left( \frac{x}{1-x} \right)^2 \) has the same explanation as in our analysis of \( \mathcal{M}_n(P_1, P_3) \).
\end{itemize}

Thus, by incorporating all these differences into \eqref{A(x)-P1-P3}, we obtain
\[
A(x) = \frac{1}{1-x} + \frac{x^2}{(1-x)^3} + x \left( \frac{x}{1-x} \right)^2 \left( 1 + \frac{2x}{1-x} \right) + \frac{x^4}{(1-x)^5},
\]
which coincides with \eqref{A(x)-P1-P2-P3}.\end{proof}

 \subsection{Triples $(P_1, P_2, P_4)$, $(P_1, P_2, P_5)$, $(P_2, P_3, P_4)$, $(P_2, P_3, P_5)$, $(P_1, P_4, P_5)$ and $(P_2, P_4, P_5)$}\label{subsec-triple-A000325} 
\begin{thm}\label{thm-A000325}
     The g.f.\ for $\M(P_1, P_2, P_4)$, $\M(P_1, P_2, P_5)$, $\M(P_2, P_3, P_4)$, $\M(P_2, P_3, P_5)$, \\ $\M(P_1, P_4, P_5)$ and $\M(P_2, P_4, P_5)$ is given by 
      \begin{align}\label{A(x)-A000325}
          A(x)=\frac{1-3x+3x^2}{(1-x)^2(1-2x)}.
     \end{align}
     The corresponding sequence is {\rm A000325} in {\rm \cite{OEIS}}.
\end{thm}
\begin{proof}
\textbf{Triples $(P_1, P_2, P_4)$ and $(P_1, P_2, P_5)$.}
Because of the reverse operation, we only need to consider \( \mathcal{M}(P_1, P_2, P_4) \). Given a matching \( M \in \mathcal{M}_n(P_1, P_2, P_4) \), the only difference from our analysis of \( \mathcal{M}(P_1, P_4) \) in Theorem~\ref{thm-A005183} is that, in the second case, the closers with their corresponding openers lying under the arc \([1, b_1]\) can appear only within the interval \([x_k, y_1]\).  Additionally, the interval $(x_k, y_1)$ and any other interval $(x_i, x_{i+1})$ for $i \in \{1, \ldots, k-1\}$ must be empty to avoid the pattern $P_2$. Thus, the contribution of this case replaces $\sum_{j\geq 0} \left(\frac{x}{1-x}\right)^j$ with $\frac{1}{1-x}$ in~\eqref{aux-eq-1}, yielding
\[
\frac{x}{1-x}\, \frac{1}{1-x} \sum_{k\geq 1}x^k\left(\frac{1}{1-x}\right)^k= \frac{x^2}{(1-x)^3\left(1-\frac{x}{1-x}\right)}.
\]
 Hence, we obtain
    \[
    A(x)=\frac{1}{1-x}+\frac{x^2}{(1-x)^3 \left(1-\frac{x}{1-x}\right)},
    \]
    which coinsides with~\eqref{A(x)-A000325}.\\[-3mm]

\noindent
 \textbf{Triples $(P_2, P_3, P_4)$ and $(P_2, P_3, P_5)$.}
Because of the reverse operation, we only need to consider $\mathcal{M}(P_2, P_3, P_4)$. Given $M\in \mathcal{M}_n(P_2, P_3, P_4)$, the differences from our analysis of $\mathcal{M}_n(P_3)$ in Section~\ref{subsec-P3-avoid} are that $(i)$ in Case 2, using Figure~\ref{case2-P3}, all the closers must be to the left of all openers between $b_1$ and $a_n$. Thus, such matchings in this case are counted by $\frac{1}{1-x}\left(\frac{x}{1-x}\right)^2$; $(ii)$ in Case 3, $C, X_2, \ldots, X_k$ can only have closers with their corresponding openers inside the arc $[1,b_1]$. Thus, $H(x)$ in~\eqref{F(x)-P3-formula} is replaced by $\frac{x}{1-x}$. Moreover, the considerations for \( B \) and \( D \) are the same as in case~$(i)$, and \( Y_i = \emptyset \) for \( 1 \leq i \leq k \); thus, \( x^2 \left(1 + \frac{x}{1 - 2x} \right)^2 \) is replaced by \( \left( \frac{x}{1 - x} \right)^2 \), and \( \left( 1 + \frac{x^2}{1 - 2x} \right) \) is replaced by 1. Therefore, substituting all the differences into~\eqref{F(x)-P3-formula}, we obtain 
\begin{align}\label{P2-P3-P4-formula}
    A(x)=\frac{1}{1-x}\left(1+\left(\frac{x}{1-x}\right)^2+\left(\frac{x}{1-x}\right)^2\frac{\frac{x}{1-x}}{1-\frac{x}{1-x}}\right).
\end{align}
Simplifying this expression yields~\eqref{A(x)-A000325}.\\[-3mm]

\noindent
 \textbf{Triple $(P_1, P_4, P_5)$.} Given $M\in \M_n(P_1, P_4, P_5)$. If $M$ is an $n$-crossing, then it contributes $\frac{1}{1-x}$ to $A(x)$, which also accounts for the empty matching. If $M$ is not an $n$-crossing, using Figure~\ref{fig-P_1-P_4}, we have the following two cases. 
 \begin{itemize}
 \item If \( I_1 \) is the only block (i.e., there is no \( I_2 \)), then each interval \( (x_i, x_{i+1}) \) or \( (x_k, y_1) \) may contain closers of arcs whose openers lie within \( [1, b_1] \). Thus, this case contributes $\frac{x}{1 - x} \sum_{k \geq 1} x^k \left( \frac{1}{1 - x} \right)^k = \frac{x^2}{(1-x)^2 \left(1 - \frac{x}{1 - x}\right)}$.
 \item If there are at least two such intervals \( I_j \)'s, then each \( I_j \) must be an \( |I_j| \)-crossing, which contributes \( \frac{x}{1 - x} \). Thus, all such \( I_j \)'s together with the first irreducible block contribute $\frac{x}{1 - x} \sum_{j \geq 2} \left( \frac{x}{1 - x} \right)^j = \frac{x^3}{(1-x)^3 \left(1 - \frac{x}{1 - x}\right)}$. \end{itemize}
 Therefore, we obtain 
 \begin{align*}
      A(x)=\frac{1}{1-x}+\frac{x^2}{(1-x)^2 \left(1 - \frac{x}{1 -x}\right)}+\frac{x^3}{(1-x)^3 \left(1 - \frac{x}{1 - x}\right)},
 \end{align*}
 which coincides with~\eqref{A(x)-A000325}.\\[-3mm]

\begin{figure}[t]
 \centering
    \begin{tikzpicture}[scale = 0.4] 
        \draw [dashed](-0.6,0) -- (6,0);
        \draw (6,0) -- (7,0);
        \draw [dashed](7,0) -- (9,0);
        \draw (9,0) -- (14,0);
        \draw [dashed](14,0) -- (16,0);
        \draw (16,0) -- (19,0);
        \draw [dashed](19,0) -- (21,0);
        \draw (21,0) -- (22,0);
        \draw [dashed](22,0) -- (26,0);
        \draw (26,0) -- (28.6,0);
        \foreach \x in {0,2,4,6,7,9,10,11,12,13,14,16,17,18,19,21,22,24,25,26,27,28}
        {
            \filldraw (\x,0) circle (2pt);
        }

        \draw[black] (0,0) arc (180:0:2 and 1.5);
        \draw[black] (6,0) arc (180:0:2 and 1.5);
        \draw[black] (7,0) arc (180:0:2 and 1.5);
        \draw[black] (9,0) arc (180:0:1.5 and 1.5);
        \draw[black] (13,0) arc (180:0:2 and 1.5);
        \draw[black] (14,0) arc (180:0:2 and 1.5);
        \draw[black] (16,0) arc (180:0:1.5 and 1.5);
        \draw[black] (21,0) arc (180:0:2.5 and 1.5);
        \draw[black] (22,0) arc (180:0:2.5 and 1.5);
        \draw[black] (24,0) arc (180:0:2 and 1.5);
        \draw[black,gray=0.1] (2,0) arc (180:0:11.5 and 2);
        \node[anchor=north west] at (8, -0.1) {\footnotesize{$I_1$}};
        \node[anchor=north west] at (14.8, -0.1) {\footnotesize{$I_2$}};
        \node[anchor=north west] at (24.2, 2.8) {\footnotesize{$I'$}};
          \node[anchor=north west] at (4.5, 1.2) {\footnotesize{$\cdots$}};
     \node[anchor=north west] at (7.5, 1.2) {\footnotesize{$\cdots$}};
     \node[anchor=north west] at (14.5, 1.2) {\footnotesize{$\cdots$}};
      \node[anchor=north west] at (19, 1.2) {\footnotesize{$\cdots$}};
     \node[anchor=north west] at (22.5, 1.2) {\footnotesize{$\cdots$}};
        \node[anchor=north west] at (-0.6, -0.1) {\scriptsize{1}};
        \node[anchor=north west] at (3.4, -0.1) {\scriptsize{$b_1$}};
        \node[anchor=north west] at (20.4, -0.1) {\scriptsize{$x_1$}};
        \node[anchor=north west] at (21.4, -0.1) {\scriptsize{$x_2$}};
        \node[anchor=north west] at (23.3, -0.1) {\scriptsize{$a_n$}};
        \node[anchor=north west] at (25.4, -0.1) {\scriptsize{$y_1$}};
        \node[anchor=north west] at (26.4, -0.1) {\scriptsize{$y_2$}};
        \node[anchor=north west] at (27.4, -0.1) {\scriptsize{$2_n$}};

    \end{tikzpicture}
\caption{The structure of $(P_2, P_4, P_5)$-avoiding Stoimenow matchings.}\label{fig-P_2-P_4_P_5}
\end{figure}
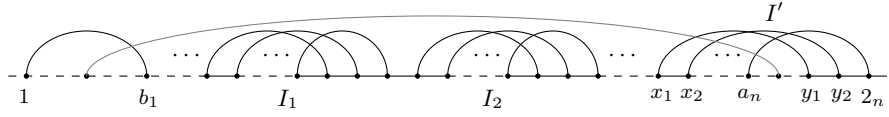

 \noindent
 \textbf{Triple $(P_2, P_4, P_5)$.} Given \( M \in \mathcal{M}_n(P_2, P_4, P_5) \), if \( M \) is an \( n \)-crossing, then it contributes \( \frac{1}{1 - x} \) to $A(x)$, which also accounts for the empty matching.If \( M \) is not an \( n \)-crossing, let \( I_i \) (\( i = 1, 2, \ldots \)) be the blocks excluding the rightmost one, defined as in the analysis of \( \mathcal{M}_n(P_1, P_4) \) in Section~\ref{subsec-P1-P4}, and let \( I' \) be the rightmost block. As shown in Figure~\ref{fig-P_2-P_4_P_5}, each \( I_j \) in \(\{ I_1, I_2, \ldots \}\) forms an \( |I_j| \)-crossing, counted by \( \sum_{j \geq 0} \left( \frac{x}{1 - x} \right)^j = \frac{1}{1 - \frac{x}{1 - x}} \), while \( I' \) contains only closers of arcs in \( (a_n, 2n) \) whose openers lie in \( [1, b_1] \) (to avoid \( P_2 \)), with this configuration counted by \( \frac{1}{1 - x} \cdot \frac{x}{1 - x} \).
 
Thus, this case, together with the crossing formed by \( [1, b_1], [2, b_1 + 1], \ldots, [m, b_1 + m - 1] \), where \( m \) is maximal, contributes
$\frac{x}{1 - x} \cdot \frac{1}{1 - x} \cdot \frac{x}{1 - x} \cdot \frac{1}{1 - \frac{x}{1 - x}} = \frac{x^2}{(1 - x)^3 \left(1 - \frac{x}{1 - x} \right)}$
to \( A(x) \). Hence, we obtain 
\begin{equation*}
    A(x) = \frac{1}{1-x} + \frac{x^2}{(1-x)^3(1-\frac{x}{1-x})},
\end{equation*}
which coincides with~\eqref{A(x)-A000325}.
\end{proof}

\subsection{Triples $(P_1, P_3, P_4)$ and $(P_1, P_3, P_5)$}

\begin{thm}\label{thm-A116725}
     The g.f.\ for $\M(P_1, P_3, P_4)$ and $\M(P_1, P_3, P_5)$ is given by 
      \begin{align}\label{A(x)-A116725}
          A(x)=\frac{1-5x+10x^2-9x^3+3x^4-x^5}{(1-x)^4(1-2x)}.
     \end{align}
     The corresponding sequence is {\rm A116725} in {\rm \cite{OEIS}}.
\end{thm}
\begin{proof}
Because of the reverse operation, we only need to consider \( \mathcal{M}(P_1, P_3, P_4) \). Let \( M \in \mathcal{M}_n(P_1, P_3, P_4) \). The only differences from our analysis of \( \mathcal{M}_n(P_3) \) in Section~\ref{subsec-P3-avoid} pertain to Case~3 and are as follows; we refer the reader to Figure~\ref{case3-P3} for the notation used.

Under the interval \( C \), there can only be closers of arcs whose openers are in \( [1, b_1] \). Thus, all arcs related to the interval \( C \) are counted by \( \frac{x}{1 - x} \). Also, each interval among \( X_1, \ldots, X_k \) (with \( k \geq 0 \)) consists of a single arc, so all such intervals are counted by \( \frac{1}{1 - x} \). Moreover, the interval \( B \) (resp., \( D \)) can only contain closers (resp., openers) of arcs whose openers (resp., closers) lie inside the arc \( [1, b_1] \) (resp., \( [a_n, 2n] \)) in order to avoid the pattern \( P_1 \). Together with the arcs \( [1, b_1] \) and \( [a_n, 2n] \), this contributes \( \left( \frac{x}{1 - x} \right)^2 \). Therefore, this case contributes
\[
\left( \frac{x}{1 - x} \right)^2 \cdot \frac{x}{1 - x} \cdot \frac{1}{1 - x} = \frac{x^3}{(1 - x)^4}
\]
to \( A(x) \). By replacing the term corresponding to Case~3 in~\eqref{F(x)-P3-formula} with \( \frac{x^3}{(1 - x)^4} \), we obtain
\[
A(x) = \frac{1}{1 - x} \left( 1 + \frac{x^2}{1 - 2x} + \frac{x^3}{(1 - x)^3} \right),
\]
which coincides with~\eqref{A(x)-A116725}. The proof is complete.
\end{proof}

\subsection{Triples $(P_3, P_4, P_5)$}

\begin{thm}\label{thm-P3-P4-P5}
     The g.f.\ for $\M(P_3, P_4, P_5)$ is given by 
       \begin{align}\label{A(x)-P3-P4-P5}
           A(x)=\frac{(1-x)^2}{1-3x+2x^2-x^3}.
      \end{align}
      The corresponding sequence is {\rm A034943} in {\rm \cite{OEIS}}.
\end{thm}
\begin{proof}
The difference between the structure of matchings in \( \mathcal{M}(P_3, P_4, P_5) \) and those in \( \mathcal{M}(P_3) \) (considered in Section~\ref{subsec-P3-avoid}) is that, in Case 3 of our analysis of $\mathcal{M}(P_3)$, there are no arcs with closers or openers in \( C, X_1, \ldots, X_k \); that is, \( C, X_1, \ldots, X_k \) are single arcs. Thus, by replacing \( H(x) \) with \( x \) in~\eqref{F(x)-P3-formula}, we obtain \( A(x) \).\end{proof}

\section{Avoidance of four patterns}\label{4-sec}

\subsection{Quadruples $(P_1, P_2, P_3, P_4)$, $(P_1, P_2, P_3, P_5)$ and $(P_2, P_3, P_4, P_5)$}

\begin{thm}\label{thm-A050407}
     The g.f.\ for $\M(P_1, P_2, P_3, P_4)$, $\M(P_1, P_2, P_3, P_5)$ and $\M(P_2, P_3, P_4, P_5)$ is given by 
      \begin{align}\label{A(x)-A050407}
          A(x)=\frac{1-3x+4x^2-x^3}{(1-x)^4}.
     \end{align}
     The corresponding sequence is {\rm A050407} in {\rm \cite{OEIS}}.
\end{thm}
\begin{proof}
\textbf{Quadruples $(P_1, P_2, P_3, P_4)$ and $(P_1, P_2, P_3, P_5)$.}
Because of the reverse operation, we only need to consider $\mathcal{M}(P_1, P_2, P_3, P_4)$. Let $M \in \mathcal{M}_n(P_1, P_2, P_3, P_4)$. The only difference from the analysis of $\mathcal{M}(P_2, P_4, P_4)$ in Section~\ref{subsec-triple-A000325} is that in Case 3 (if there are at least one arc between $b_1$ and $a_n$), as shown in~\ref{case3-P3}, the closers of arcs with openers inside the arc $[1, b_1]$ can only appear in the intervals $B, C$. Such intervals, together with the arc $[1, b_1]$, give $\frac{x}{1-x} \cdot \frac{x}{1-x} = \left( \frac{x}{1-x} \right)^2$. Thus, the term $\frac{1}{1-x} \left( \frac{x}{1-x} \right)^2 \cdot \frac{x}{1-2x}$ in~\eqref{P2-P3-P4-formula} is replaced by $\frac{x^3}{(1-x)^4}$, where $X_1, \ldots, X_k$ are single arcs contributing $\frac{1}{1-x}$ and the last irreducible block contributing $\frac{x}{1-x}$. Hence, we obtain
\[
A(x)=\frac{1}{1-x}\left(1+\left(\frac{x}{1-x}\right)^2+\left(\frac{x}{1-x}\right)^3\right),
\]
which coincides with~\eqref{A(x)-A050407}.\\[-3mm]

\noindent
\textbf{Quadruple $(P_2, P_3, P_4, P_5)$.} Let $M\in \M_n(P_2, P_3, P_4, P_5)$. If $M$ is an $n$-crossing, then it contributes $\frac{1}{1-x}$ to $A(x)$, which also accounts for the empty matching. If $M$ is not an $n$-crossing, then referring to Figure~\ref{case3-P3} in Section~\ref{subsec-P3-avoid}, the intervals $B$ and $D$ contain only closers of arcs with openers inside the arc $[1, b_1]$ and openers of arcs with closers inside the arc $[a_n, 2n]$, respectively. The intervals $B$ and $D$, along with the arcs $[1, b_1]$ and $[a_n, 2n]$, contribute $\left( \frac{x}{1-x} \right)^2$. Meanwhile, there exist, possibly, arcs with openers under $[1, b_1]$ and closers under $[a_n, 2n]$, which contribute $\frac{1}{1-x}$. Moreover, the arcs between $b_1$ and $a_n$ form an $\ell$-noncrossing, where $\ell$ represents the number of such arcs, counted by $\frac{1}{1-x}$. Thus, this case totally contributes $\frac{x^2}{(1-x)^4}$ to $A(x)$. Thus, we obtain
\[
A(x)= \frac{1}{1-x}+\frac{x^2}{(1-x)^4},
\]
which coincides with~\eqref{A(x)-A050407}. The proof is complete.
\end{proof}

\subsection{Quadruples $(P_1, P_2, P_4, P_5)$ and $(P_1, P_3, P_4, P_5)$}

\begin{thm}\label{thm-P1-P2-P4-P5, P1-P3-P4-P5}
     The g.f.\ for $\M(P_1, P_2, P_4, P_5)$ and $\M(P_1, P_3, P_4, P_5)$ is given by 
      \begin{align}\label{A(x)-P1-P2-P4-P5, P1-P3-P4-P5}
          A(x)=\frac{1-4x+6x^2-3x^3-x^4}{(1-x)^3(1-2x)}.
     \end{align}
     
\end{thm}
\begin{proof}
\textbf{Quadruple $(P_1, P_2, P_4, P_5)$.} Let $M \in \mathcal{M}_n(P_1, P_2, P_4, P_5)$. If $M$ is an $n$-crossing, then it contributes $\frac{1}{1-x}$ to $A(x)$, which also accounts for the empty matching. If $M$ is not an $n$-crossing and there is no arc between $b_1$ and $a_n$, then there may exist closers of arcs with openers inside the arc $[1, b_1]$ and openers of arcs with closers inside the arc $[a_n, 2n]$ without overlapping. Meanwhile, there may exist arcs with openers inside $[1, b_1]$ and closers inside $[a_n, 2n]$. Thus, this case totally contributes $\left( \frac{x}{1-x} \right)^2 \cdot \frac{1}{1-x} = \frac{x^2}{(1-x)^3}$ to $A(x)$. If there is at least one arc between $b_1$ and $a_n$, then each irreducible block, denoted by $I_j$, of $M$ is formed by an $|I_j|$-crossing (note that there are at least three irreducible blocks), which contributes $\sum_{j \geq 3} \left( \frac{x}{1-x} \right)^j = \frac{x^3}{(1-x)^2(1-2x)}$ to $A(x)$ in total. Thus, we obtain 
\[
A(x)=\frac{1}{1-x}+\frac{x^2}{(1-x)^3}+\frac{x^3}{(1-x)^3 \left(1-\frac{x}{1-x}\right)},
\]
which coincides with~\eqref{A(x)-P1-P2-P4-P5, P1-P3-P4-P5}.

\noindent
\textbf{Quadruple $(P_1, P_3, P_4, P_5)$.} Let $M \in \mathcal{M}_n(P_1, P_3, P_4, P_5)$. The only difference from our analysis of $\mathcal{M}_n(P_3)$ is that in Case 3 (there is at least one arc between $b_1$ and $a_n$), the intervals $B$ (resp., $D$) can only have closers (resp., openers) of arcs with openers (resp., closers) inside the arc $[1, b_1]$ (resp., $[a_n, 2n]$). The intervals $B$ and $D$, along with the arcs $[1, b_1]$ and $[a_n, 2n]$, contribute $\left( \frac{x}{1-x} \right)^2$. Meanwhile, arcs between $b_1$ and $a_n$ form an $\ell$-noncrossing with $\ell \geq 1$ (the number of such arcs), contributing $\frac{x}{1-x}$. Thus, this case totally contributes $\left( \frac{x}{1-x} \right)^3$. Therefore, substituting the distinct factors in~\eqref{F(x)-P3-formula}, we obtain
\[
A(x)=\frac{1}{1-x}\left(1+\frac{x^2}{1-2x}+\frac{x^3}{(1-x)^2}\right),
\]
which coincides with~\eqref{A(x)-P1-P2-P4-P5, P1-P3-P4-P5}. The proof is complete.
\end{proof}

\section{Avoidance of five patterns}\label{5-sec}

\begin{thm}\label{thm-P1-P2-P3-P4-P5}
     The g.f.\ for $\M(P_1, P_2, ,P_3, P_4, P_5)$ is given by 
      \begin{align}\label{A(x)-P1-P2-P3-P4-P5}
          A(x)=\frac{1-2x+2x^2+x^3}{(1-x)^3}.
     \end{align}
     The corresponding sequence is {\rm A002522} in {\rm \cite{OEIS}}.
\end{thm}
\begin{proof}
   Let $M \in \mathcal{M}_n(P_1, P_2, P_3, P_4, P_5)$. The items of Case 1 ($M$ is an $n$-crossing) and Case~3 (there is at least one arc between $b_1$ and $a_n$) are the same as those in the analysis of the set $\mathcal{M}_n(P_1, P_3, P_4, P_5)$, while in Case 2 ($M$ is not an $n$-crossing and there are no arcs between $b_1$ and $a_n$), closers and openers between $b_1$ and $a_n$ cannot be shuffled, which gives $\left( \frac{x}{1-x} \right)^2 \cdot \frac{1}{1-x} = \frac{x^2}{(1-x)^3}$ in this case, where $\frac{1}{1-x}$ represents arcs with openers inside the arc $[1, b_1]$ and closers inside the arc $[a_n, 2n]$. Thus, we obtain
    \[
    A(x)=\frac{1}{1-x}\left(1+\left(\frac{x}{1-x}\right)^2+\frac{x^3}{(1-x)^2}\right),
    \]
    which coincides with~\eqref{A(x)-P1-P2-P3-P4-P5}. We complete the proof.
\end{proof}


\begin{thebibliography}{9}

\bibitem{Bevan-Cheon-Kitaev}  D. Bevan, G. Cheon and S. Kitaev. 
On naturally labelled posets and permutations avoiding 12-34. 
{\it European J. Combin.} {\bf 126} (2025) 104117.

\bibitem{Bousquet-Claesson-Dukes-Kitaev-DukesParvi} M. Bousquet-M\'elou, A. Claesson, M. Dukes and S. Kitaev. $\mathbf{(2+2)}$-free posets, ascent sequences and pattern avoiding permutations. {\it J. Comb. Theory, Ser. A.} {\bf 117} (2010) 884--909.

\bibitem{Claesson-Dukes-Kitaev} A. Claesson, M. Dukes and S. Kitaev. A direct encoding of Stoimenow’s matchings as ascent sequences. {\it Australas. J. Comb.}  {\bf 49} (2011) 47--59.

\bibitem{DisantoPerPinRin} F. Disanto, E. Pergola, R. Pinzani and S. Rinaldi. Generation and Enumeration of Some Classes of Interval Orders, {\it Order} {\bf 30} (2013) 663--676.

\bibitem{DukPar} M. Dukes and R. Parviainen. Ascent sequences and upper triangular matrices containing nonnegative integers. {\it Electron. J. Combin.} {\bf 17} (2010), \#P53.

\bibitem{Fishburn} P.C. Fishburn. Intransitive indifference with unequal indifference intervals. {\it J. Math. Psych.} {\bf 7} (1970) 144--149.

\bibitem{OEIS} OEIS Foundation Inc., {\it The {O}n-{L}ine {E}ncyclopedia of {I}nteger {S}equences},
published electronically at \phantom{*}{\tt
https://oeis.org}.

\bibitem{LvKitZhang} S. Lv, S. Kitaev and P. B. Zhang. Catalan structures arising from pattern-avoiding subclasses of Stoimenow matchings and other Fishburn objects. ArXiv:2509.09115. 

\bibitem{stoim} A. Stoimenow. Enumeration of chord diagrams and an upper bound for Vassiliev invariants. {\it J. Knot Theory Ramifications} {\bf 7} (1998) 93--114.

\end{thebibliography}
\end{document}